\documentclass[reqno]{amsart}
\usepackage{xcolor}         
\usepackage{amsmath}
\usepackage{amsthm}
\usepackage{enumitem}
\usepackage{algorithm}
\usepackage{algorithmic}
\usepackage{graphicx}
\newtheorem{lemma}{Lemma}
\newtheorem{remark}{Remark}
\newtheorem{theorem}[]{Theorem}
\newtheorem{definition}[]{Definition}
\newtheorem{prop}{Proposition}

\let\pa\partial

\renewcommand{\L}{\mathcal{L}}

\usepackage{geometry}
\usepackage{parskip}
\usepackage[normalem]{ulem}

\title[Inf-Sup Networks for PDEs]
{Inf-Sup neural networks for high-dimensional elliptic PDE problems}
\author{Xiaokai Huo}
\email{xhuo@iastate.edu}
\author{Hailiang Liu}
\email{hliu@iastate.edu}
\address{Department of Mathematics, Iowa State University, Ames, IA 50014}

\begin{document}
\maketitle 

\begin{abstract}
Solving high dimensional partial differential equations (PDEs) has historically posed a considerable challenge when utilizing conventional  numerical methods, such as those involving domain meshes. 
Recent advancements in the field have seen the emergence of neural PDE solvers, leveraging deep networks to effectively tackle high dimensional PDE problems.  This study introduces Inf-SupNet, a model-based unsupervised learning approach designed  to acquire solutions for a specific category of elliptic PDEs. The fundamental concept behind Inf-SupNet involves incorporating 
the inf-sup 
formulation of the underlying PDE into the loss function. 
The analysis reveals that the global solution error can be bounded by the sum of three distinct errors: the numerical integration error, the duality gap of the loss function (training error), and the neural network approximation error for functions within Sobolev spaces. To validate the efficacy of the proposed method, numerical experiments conducted in high dimensions demonstrate its stability and accuracy across various boundary conditions, as well as for  both semi-linear and nonlinear PDEs.
\end{abstract}

\section{Introduction}

In recent years, the burgeoning field of deep neural networks (DNNs) has propelled the widespread use of machine learning techniques in scientific computing. 
Renowned for their status as universal function approximators, 
DNNs have gained popularity as ansatz spaces for approximating solutions to diverse  partial differential equations(PDEs)    
\cite{yu2018deep,yu2022gradient,raissi2017physics,zang2020weak,lagaris1998artificial,rudd2015constrained,han2017deep,beck2019machine}. In contrast to conventional numerical methods such as finite elements (FEM), finite differences (FDM), or finite volumes (FVM), which rely on grids or meshes, neural networks offer a mesh-free, scalable approach that holds promise for solving high-dimensional PDEs. In 
scenarios where ample labeled training data can be obtained through specific numerical methods, supervised learning becomes a viable option for training DNNs to approximate PDE solutions, as exemplified in  \cite{raissi2017inferring,raissi2017machine,owhadi2015bayesian,KLY2017,kutyniok2022theoretical}.
However, in many contexts, a learning framework is required to approximate PDE solutions with minimal or potentially no data.    
In such situations, the underlying PDE and associated conditions must be seamlessly integrated  into an unsupervised learning strategy \cite{lagaris1998artificial,raissi2017physics}. The method presented here falls into this category.

The idea of incorporating PDEs and associated conditions with neural networks for inferring PDE solutions can be traced back to a 1994 paper \cite{dissanayake1994neural}. In this early work,  neural networks were limited to one hidden layer. Subsequently, a similar idea was explored by  \cite{lagaris1998artificial}, 
where boundary conditions were manually enforced through a change of variables. These initial concepts evolved into more sophisticated approaches such as  Physics-informed neural networks (PINNs) \cite{raissi2017physics} and the deep Galerkin method \cite{sirignano2018dgm},
both featuring deeper hidden layers. Among these,  PINNs have gained popularity in the field  of scientific machine learning in recent years. 
For a differential equation $\mathcal{A}(u)=f$ defined on the domain $\Omega$ with boundary condition $\mathcal{B} u=g$ on $\partial \Omega$, PINN trains a neural network to minimize the following functional when evaluated on sampling points, 
$$
{\L}(v)=\int_{\Omega} |\mathcal{A}(v)(x)-f|^2dx+\lambda \int_{\partial \Omega} |\mathcal{B}(v)(x)-g|^2 ds,
$$
where $\lambda>0$ is a weight parameter. However, minimizing such a functional may lead to a higher order PDE than the original one, posing increased difficulty when the solution regularity is relatively low. As an improvement, the deep Least-Squares method introduced in \cite{CCLL2020} employs  a least-squares functional based on the first order system of scalar second-order elliptic PDEs, treating  boundary conditions in a balanced manner.  Various  modifications have been proposed to enhance the accuracy of PINNs or adapt them for specific problems.  Examples include the gradient-enhanced PINNs (gPINNs), which  incorporate higher order solution derivatives \cite{yu2022gradient}, and the Variational PINNs (VPINNs) \cite{kharazmi2019variational}, which incorporate the weak form of the PDE. For a comprehensive overview of PINNs and their applications across diverse fields, the reader is directed to \cite{cai2021physics,mahmoudabadbozchelou2022nn,mao2020physics,haghighat2021physics,rasht2022physics,yang2021b}. 
 
Various advancements have been made in formulating loss functions for solving PDE problems with deep neural networks (DNNs). The deep Ritz method (DRM) \cite{yu2018deep}  employs the energy functional of the Poisson equation and is particularly applicable to problems grounded in a minimization principle. 
Variational neural networks (VarNets) \cite{khodayi2020varnet} utilize the variational (integral) form of advection-diffusion equations, with test functions sampled from linear spaces.  Using lower-order derivatives with training over space-time regions, the method using VarNets gets more effective in capturing the PDE solutions.
In the context of hyperbolic conservation laws and the pursuit of entropy solutions, weak PINNs (wPINNs) were introduced, incorporating the Kruzkov entropy, as analyzed in \cite{de2022weak}. Another notable approach is the Weak Adversarial Networks (WANs) \cite{zang2020weak}, which also leverage the weak PDE formulation. However, WANs introduce an additional adversarial network to represent the test function, enhancing the overall robustness and efficacy of the method.
 
A notable limitation of weak PDE formulations,  such as the Weak Adversarial Networks (WANs), is their intrinsic lack of numerical stability,  as highlighted in \cite{D2RM}. Addressing this issue, the Deep Double Ritz method, introduced in the same reference \cite{D2RM}, adopts a nested Ritz minimization strategy to optimize test functions and enhance stability. In this paper, we propose a simple alternative, namely Inf-SupNet. 
Our approach originates from a constraint optimization problem:
$$
\inf_{u\in X} \int_{\partial \Omega} |\mathcal{B}u(x)-g|^2 ds  \quad \text{subject to} \quad \mathcal{A} u=f \quad \text{in} \quad \Omega,  
$$
which is demonstrated to be equivalent to the aforementioned PDE problem. In such formulation, the least square of the boundary data mismatch is the primary loss, with the underlying PDE serving as a constraint. This constraint is elevated to  the primal-dual space as an inf-sup problem:
 $$
 \inf_{u\in X}\sup_{v\in Y^*}
 \|\mathcal{B}u(x)-g\|_{\partial \Omega}^2 +( \mathcal{A} u-f, v)_{\Omega},
 $$
 where $\mathcal{A}:X\to Y$ represents the underlying differential operator, and $Y^*$ is the dual space of $Y$. InfSupNet employs this loss function, leveraging  the inf-sup form of the PDE problem and employing  an adversarial network for the Lagrangian multiplier $v$.  Our method provides  flexibility in handling the dual variable $v$ based on the solution regularity of $u$. 

Inspired by a recent work \cite{liu2023primal} combining primal-dual hybrid gradient methods with spatial Galerkin methods for approximating conservation laws with implicit time-discretization, our formulation differs as we introduce neural networks to represent the trial solution and the Lagrangian multiplier. This renders the resulting optimization problem non-quadratic in the parameter space.

Regarding the theoretical accuracy of the InfSupNet algorithm, we derive an error estimate based on the stability of elliptic problems. The error is shown to be bounded by the numerical integration error, the neural network approximation error, and the training error. 
The  min-max structure facilitates the quantification of the training error through the duality gap between the loss for minimization and maximization steps.

{\bf Contribution.} 
Our key contributions for learning solutions of PDE problems are as follows. 
\begin{itemize}
    \item Present an InfSupNet, a novel method for learning solutions of PDE problems using deep neural networks. 
    \item Prove equivalence of the inf-sup problem to the original PDE problem. 
    \item Show the global approximation error is induced by three sources:  the training error, the sampling error, and the approximation error of neural networks; and each error is shown to be asymptotically small under reasonable assumptions.
    \item Demonstrate through experiments both stability and accuracy  of the method for high-dimensional elliptic problems posed on both  regular and irregular domains.
\end{itemize}

\subsection*{Related work}  Here we delve into further
discussions on related works,  encompassing neural PDE solvers, sampling methods, and  error estimation.

\subsection*{Mesh Dependent Neural Methods:} 
Our method, along with other mesh-free approaches like  PINNs, DRM, WANs, relies on auto-differentiation to compute the empirical loss, 
eliminating the need for a mesh grid. This enables obtaining solutions at new locations beyond the sampling points directly after the network is trained. 
In contrast, mesh-based methods, as seen in  \cite{khara2021neufenet,khara2022neural,khoo2021solving,bhatnagar2019prediction}, employ  conventional numerical techniques such as finite difference or finite element methods, heavily relying on mesh grids for computations.  
While mesh-based methods permit the use of equation parameters on grids as input for network training and solving a range of PDEs with varying parameters, changing the mesh resolutions  alter network structures, and 
this will necessitate new training. In contrast, mesh-free methods maintain the underlying network when changing the resolution of numerical integration, making them more adaptable, especially in high dimensions 
\cite{cuomo2022scientific}.

\subsection*{Deep BSDE Solvers:}  
Establishing a connection between parabolic PDEs and backward stochastic differential equations (BSDEs), 
\cite{han2017deep} introduces a deep BSDE method to evaluate solutions at any given space-time locations for high dimensional semi-linear parabolic PDEs. This method is extended to handle fully nonlinear second-order PDEs in \cite{beck2019machine}.

\subsection*{Operator Learning:} A novel approach involves learning infinite dimensional operators with neural networks 
\cite{lu2021learning,li2020fourier}. Universal approximation theorems for operator learning have been established \cite{chen1993approximations}, operator learning allows the solution of  a family of PDEs with one neural construction. This approach operates as a form of supervised learning, requiring only data, devoid of prior knowledge about the underlying PDE.  

\subsection*{Data Sampling Techniques:}  Apart from utilizing deeper hidden layers for solving PDEs, as discussed earlier, it is also illustrated that the training points can be obtained by a random sampling of the domain rather than using a mesh, which is beneficial in higher-dimensional problems \cite{BN2018, sirignano2018dgm}. 
Indeed,  different numerical integration schemes are employed in all mesh-free methods. For an in-depth comparison study of quadrature rules used in solving PDEs with deep neural networks, please refer to \cite{RTOP22}.

\subsection*{Error Bounds:} Although not as mature as the applications of neural PDE solvers in various domains, the theoretical development has made significant strides, presenting rigorous bounds on various error sources 
\cite{shin2023error,mishra2022estimates,mishra2023estimates,jiao2021error,minakowski2023priori,zeinhofer2023unified}. Typically total error is decomposed into the approximation error of neural networks, numerical sampling error and training error, each being bounded through  different strategies.  
For PINNs, as identified in \cite{mishra2023estimates}, the strategy involves considering regularity and stability of PDE solutions 
and quadrature error bounds to estimate the generalization gap between continuous and discrete versions of the PDE residual. 
In deriving error bounds for InfSupNet, we pursue a related yet distinct approach. Beyond the error on PDE solutions, we also estimate the error of the Lagrangian multiplier. Moreover, due to the use of min-max optimization,
the training error is no longer solely quantified by the loss function but rather by the duality gap of the loss function.

\subsection*{Organization} The paper is organized as follows. In the next section, we describe our method. First, we reformulate the PDE problem as a constraint optimization in Section 2.1, and discuss the network structure in Section 2.2. Our algorithm is given in Section 2.3. Section 3 is devoted to the theoretical analysis of our algorithm. We prove the equivalence between the optimization problem and the PDE problem in Section 3.1, Theorem \ref{thm1}. In Section 4, we derive an error estimate for the proposed algorithm, with a decomposition of the error in Section 4.1, Theorem \ref{thm2}, with the bounds of each part provided in Section 4.2. We demonstrate in several numerical experiments the  accuracy
of our algorithm in Section 5.
Finally, some concluding remarks are given in Section 6. 

\noindent\textbf{Notations.}
We take $L^2$ to be the square integrable space with norm defined by  $\|f\|_{L^2(\Omega)} = ( \int_{\Omega} |f(x)|^2 dx)^{1/2} $ with $\Omega \subset \mathbb{R}^d$ being an open domain of $\mathbb{R}^d$. The boundary set of the domain $\Omega$ is denoted by $\partial\Omega$. We use $H^s(\Omega)$  to denote the Sobolve space with the up-to $s$-th derivatives bounded in $L^2$.
We use $\mathcal{U}_\theta$ to denote a U-network, and $\mathcal{V}_\tau$ as  \text{V-network}, with $\theta, \tau$ as respective network parameters.  
For a positive integer $m$, we use $[m]$ to represent $\{1, \cdots, m\}$.

\section{Proposed Method}

We initiate by presenting an abstract form of the PDE problem, using the second order  linear elliptic equations as a canonical example. 

\subsection{Reformulation as constrained optimization} 
Consider solving a boundary value problem  of the form 
\begin{equation}\label{eq}
\begin{aligned}
    \mathcal{A} u &=f ,\quad \text{in }\Omega,\\
    \mathcal{B} u &=g,\quad \text{on }\partial\Omega,
\end{aligned}
\end{equation}
where $\mathcal{A}$ is a differential operator with respect to $x$,  $\mathcal{B}$ is a boundary operator which represents Dirichlet, Neumann, Robin or mixed boundary conditions. Here $\Omega$ is a bounded domain $\Omega\subset \mathbb{R}^d$, for dimension $d\ge 1$, and
$f=f(x),g=g(x)$ are given functions defined in $\Omega$ and on the boundary $\partial\Omega$, respectively. Here $u=u(x):\Omega  \mapsto\mathbb{R}$ is the unknown function we want to seek for. To proceed, we assume that the above boundary value problem is well-posed.    
 Here as an example, we take
 \begin{align}\label{eq:el}
     \mathcal{A} u := - \sum_{i=1}^d \left( \sum_{j=1}^d a_{ij}(x)u_{x_j}\right)_{x_i} + \sum_{i=1}^d b_{i}(x) u_{x_i} + c(x) u,
 \end{align}
 and 
 \begin{align}\label{eq:elB}
     \mathcal{B} u:= \alpha u + \beta \frac{\partial u}{\partial n},
 \end{align}
 with $\alpha>0$, $\beta\ge 0$, and $\frac{\partial u}{\partial n}$ denoting the directional derivative of $u$ along the exterior normal direction $n=n(x)$ at the boundary point $x\in \partial \Omega$.  Here $a_{ij}, b_i, c_i$ are bounded functions defined in $\Omega$ and $a_{ij}$ are differentiable and satisfy $\sum_{i,j=1}^d a_{ij}(x) \xi_i\xi_j \ge \lambda \|\xi\|^2$ for some constant $\lambda>0$ for any $\xi\in\mathbb{R}^d$ and $x\in\Omega$, where $a_{ij}=a_{ji}$ for all $i, j\in [d]$ (ellipticity). For the above elliptic equations, we assume { $f\in L^2(\Omega), g\in H^{3/2}(\partial\Omega)$}, so the corresponding boundary value problem is well-posed with solution in $H^2(\Omega)$ (see for example \cite{grisvard2011elliptic}).

To solve \eqref{eq},  we reformulate it as a constraint optimization problem
\begin{equation}
    \inf_{u\in U} \left\{   \frac12 \|\mathcal{B} u -g\|_{L^2(\partial\Omega)}^2\right\},\quad \text{where } 
    U=\{ u\in H^2(\Omega), \quad \mathcal{A} u =f \;\; \text{ in }\Omega\}. 
    \label{eq:co}
\end{equation}
In this formulation, the objective function is the least square of the boundary data mismatch, and the PDE itself serves as a constraint. The equivalence between this optimization problem and equation \eqref{eq} is obvious, and existence of problem  (\ref{eq:co}) will be established in the next section.

To remove the PDE constraint, we introduce a Lagrangian multiplier $v\in L^2(\Omega)$, so that the above problem can be converted to an inf-sup optimization problem
\begin{equation}
    \inf_{u\in H^2(\Omega)}\sup_{v\in L^2(\Omega)}  \mathcal{L}(u,v):=  \frac12 \|\mathcal{B} u -g\|_{L^2(\pa\Omega)}^2 + (\mathcal{A} u -f ,v),\label{eq:is}
\end{equation}
where $(\cdot,\cdot)$ denotes the $L^2$ inner product in $\Omega$. Here the loss function is made up of two terms: the first term takes into consideration of the boundary condition and the second term accounts for the underlying PDE. The saddle point of the above Lagrangian can be shown to be the solution to problem \eqref{eq}, see Theorem \ref{thm1} in the next section. 

\subsection {Neural networks}  
Here we use two feedforward  neural networks (FNNs) to represent the trial solution and the Lagrangian multiplier. 
Given input $x$, an FNN 
transforms it to an output, which can be seen as a recursively defined function $u_\theta$ on  $\mathbb{R}^d$ into $\mathbb{R}$:
$$
u_\theta(x) = h_m \circ \sigma \cdots h_2 \circ  \sigma \circ  h_1(x).
$$
Here $h_k$ is an affine map from $\mathbb{R}^{N_{k-1}}$ to $\mathbb{R}^{N_k}$ (here we take $N_0=d, N_m=1$):
$$
h_k(z) = W_k z+b_k, \quad \theta=\{W_k, b_k\}_{k=1}^m,
$$
where $W_k$ are $N_k \times N_{k-1}$ matrices (network weights) and $b_k$ are $N_k$--vectors (network biases). 
$\sigma$ is a nonlinear activation function. As the present method can be combined with other neural network constructions, FNN serves a base model class to demonstrate our method. Moreover, by the universal approximation theorem of FNNs \cite{cybenko1989approximation, pinkus1999approximation,barron1993universal}, with the use of a proper activation, FFNs of sufficient width and depth can model any given $H^2$ and $L^2$ 
function which covers our solution space and the function space of the Lagrangian multiplier. Therefore, we use FFNs to model both trial solutions and the Lagrangian multiplier.
Required by assumptions of the universal approximation theorems (\cite{cybenko1989approximation, pinkus1999approximation,barron1993universal}), the activation function $\sigma$ can be taken as any differentiable non-constant bounded function. Examples include sigmoid function and $\tanh$.

\subsection {The Inf-Sup network}
The inf-sup reformulation of the PDE problem led us to design the inf-sup network, which consists of  a neural network $\mathcal{U}_\theta$ representing the trial solution and another adversarial network $\mathcal{V}_\tau$ representing the Lagrangian multiplier. The subscripts indicate the parameters of each neural network adopted. We can thus approximate problem \eqref{eq:is} by 
\begin{align} \label{loss}
   \min_\theta \max_\tau L(\theta,\tau):=  \mathcal L(u_\theta,v_\tau),
\end{align}
with $u_\theta \in \mathcal U_\theta$ and $v_\tau \in \mathcal V_\tau$. The differential operator $\mathcal A$ and $\mathcal B$ in $\mathcal L$ are computed via auto-differentiation of neural networks.

With the above setup, the approximation solution to the inf-sup optimization problem \eqref{loss} can be obtained by finding a saddle point for  the following min-max problem:  
\begin{align}
    \min_{\theta} \max_{\tau} { \hat{L}}(\theta,\tau), 
     \label{eq:isn}
\end{align}
where  $\hat L $ denotes the evaluation of ${\mathcal L}(u_\theta,v_\tau)$ over samples
using the Monte-Carlo integration  \cite{robert1999monte}: 
\begin{align}\label{eq:hatLL}
    \hat L(\theta,\tau) := \hat {\mathcal L}(u_\theta,v_\tau)= &\frac{|\partial\Omega|}{2N_b}\sum_{i=1}^{N_b} (\mathcal{B} u_{\theta}(x_b^i)-g(x_b^i))^2 
    + \frac{|\Omega|}{N} \sum_{i=1}^N (\mathcal{A} u_{\theta}(x^i) - f(x^i))\cdot v_\tau (x^i).
\end{align}
Here $N_b$ is the number of sampling points $\{x_b^i\}_{i=1}^{N_b}$ on the boundary $\partial \Omega$ and $N$ is the number of sampling points $\{x^i\}_{i=1}^N$ inside the domain $\Omega$. 
Note that other numerical integration methods could also be used.

 Problem \eqref{eq:isn} can be solved using techniques of the min-max optimization such as Stochastic Gradient Descent--Ascent \cite{razaviyayn2020nonconvex} .  
The algorithm is described in Algorithm \ref{alg:AGF}.
\begin{algorithm}[ht]
\caption{\small Training of InfSupNet}
\begin{algorithmic}
\label{alg:AGF}
\FOR{number of training iterations}
    \STATE{$\bullet$ Sample $N$ points $\{x^1,\dots,x^N\}$ randomly in the domain $\Omega \subset \mathbb{R}^d$.}
    \STATE{$\bullet$ Sample $N_b$ points $\{ {x}_b^{(1)}, \dots, {x}_b^{(N_b)} \}$ randomly from  the boundary set $\partial\Omega \subset \mathbb{R}^{d-1}$.}
  \FOR{$k$ steps}
    \STATE{$\bullet$ Update the $\mathcal{V}_\tau$-network by ascending its (stochastic) gradient:
        \[
        \tau \gets \tau + \eta \nabla_\tau \hat{L}(\theta,\tau)
        \]}
   \ENDFOR
   \FOR{$\ell$ steps}
    \STATE{$\bullet$ Update the $\mathcal U_\theta$-network by descending its (stochastic) gradient:
        \[
        \theta\gets \theta-\eta \nabla_\theta \hat L(\theta,\tau)
        \]}
    \ENDFOR
  \ENDFOR

  \RETURN { $\theta, \tau$ for the parameters of the neural network $\mathcal U_\theta,\mathcal V_\tau$.}
\end{algorithmic}
\end{algorithm}

\section{Theoretical results}
In this section, we provide theoretical analysis for our proposed method.
For simplicity of presentation, we present  analysis and the result only for the Poisson equation with $\mathcal{A} = -\Delta$ being the negative Laplacian and $\mathcal{B}=I$ being the identical map. The results and analysis extend well to problem \eqref{eq:el}-\eqref{eq:elB} in a straightforward fashion. The Poisson problem reads as 
\begin{equation}\label{eq:1}  
\begin{aligned}
    -\Delta u = f, &\quad \text{in } \Omega, \\
     u = g,&\quad \text{on } \partial \Omega.
\end{aligned}
\end{equation}

\subsection{ Equivalence between the inf-sup problem \eqref{eq:is} and the PDE formulation}
In this subsection, we prove the following theorem.
\begin{theorem}\label{thm1}
  Assume $\Omega$ is a bounded domain with $\partial\Omega$ being smooth, and  $f\in L^2(\Omega), { g\in H^{2/3}(\partial\Omega)}$. Then $u$ is the unique solution to \eqref{eq:1} if and only if it solves the optimization problem \eqref{eq:is} with $\mathcal B u= u$ and $\mathcal A u = -\Delta u$.    
\end{theorem}
In order to prove the above theorem and show the equivalence between the inf-sup optimization problem \eqref{eq:is} and problem \eqref{eq:1},  we utilize the nice property of the saddle points.

\begin{definition}
A pair of functions $(u,v) \in H^2(\Omega)\times L^2(\Omega)$ is called a \emph{saddle point} of the Lagrangian functional $\mathcal L(u,v)$ if the following inequality holds
\begin{align} \label{eq:lsa}
    \mathcal L(u,q) \le \mathcal  L(u,v)\le \mathcal  L(w,v)\quad \text{for any } w\in H^2(\Omega), q \in L^2(\Omega).
\end{align}
\end{definition}

By Lemma \ref{lem4} in the appendix, a pair of function $(u^*,v^*)$ is a saddle point of the Lagrangian $\mathcal L$ if and only if the inf-sup problem \eqref{eq:is} satisfies the \emph{strong max-min property} 
\begin{align} \label{eq:dual}
    \inf_{u\in H^2(\Omega)} \sup_{v\in L^2(\Omega)} \mathcal  L(u,v) = \sup_{v\in L^2(\Omega)} \inf_{u\in H^2(\Omega)} \mathcal  L(u,v),
\end{align}
and $u^*$ and $v^*$  are the corresponding primal and dual optimal solutions. 
Therefore, to show the equivalence between problem \eqref{eq:is} and the PDE problem \eqref{eq:1}, it suffices to show the max-min property, and that the solution to \eqref{eq:1} and the saddle point of $\mathcal L$ are equivalent.

\subsubsection*{The strong max-min property} 
Suppose $u$ is the optimal solution to \eqref{eq:is}, by the linearity of the constraint, we have 
\begin{align} \label{eq:ll1}
    \sup_{v\in L^2(\Omega)} \mathcal L(u,v)  = \left\{\begin{array}{cc}
          +\infty, & \text{if } -\Delta u -f \neq 0, \\
         \frac12 \|u-g\|_{L^2(\partial\Omega)}^2, &\text{otherwise}.
    \end{array} \right.
\end{align}
Since for $f,g$ satisfying the assumption of Theorem \ref{thm1},  \eqref{eq:1} has a unique solution,  $\inf_{u\in H^2(\Omega)}   \sup_{v\in L^2(\Omega)} \mathcal L(u,v) = 0$. On the other hand, 
\begin{align}\label{eq:ll2}
        \inf_{u \in L^2(\Omega)} \mathcal L(u,v)  = \left\{\begin{array}{cc}
          -\infty, & \text{if } v\neq 0, \\
         0, &\text{otherwise},
    \end{array} \right.
\end{align}
since for $v\neq 0$, one can take $-\Delta u - f = - \lambda v $ and $u=g$ on $\partial\Omega$ so that $\mathcal L(u,v)\to -\lambda \|v\|^2$, and let $\lambda \to \infty$ to get $\mathcal L(u,v)\to -\infty$. Therefore,  $$\sup_{u\in H^2(\Omega)}   \inf_{v\in L^2(\Omega)} \mathcal L(u,v) = 0 = \inf_{u\in H^2(\Omega)}   \sup_{v\in L^2(\Omega)} \mathcal L(u,v)$$ and the strong max-min property holds.

\subsubsection*{Equivalence between the saddle point and PDE} 
First we derive the equation of the saddle point. Suppose $(u,v)$ is a saddle point of $\mathcal{L}$, then by definition \eqref{eq:lsa}, $\mathcal L(u,v+\varepsilon q) \le \mathcal L(u,v) \le \mathcal L(u+\varepsilon w,v)$ for any $q \in L^2(\Omega)$, $w\in H^2(\Omega)$ and $\varepsilon \in \mathbb{R}$. By the arbitrariness of $\varepsilon$, one can derive the following equations:  
\begin{equation}\label{eq:w1}
\begin{aligned}
    &\int_{\partial\Omega} u w ds + \int_\Omega (-\Delta w) v dx = \int_{\partial\Omega} w g ds,\quad \text{for any } w\in H^2(\Omega), \\ 
    &\int_\Omega (-\Delta u - f) q dx = 0,\quad \text{for any }q\in L^2(\Omega).
\end{aligned}
 \end{equation}
 Note that the converse is also true. If $(u,v)$ satisfies the above equations, we can compute $\mathcal L(u,q)- \mathcal L(u,v) = \int_\Omega (-\Delta u -f)(q-v)dx$, which by the second equation of \eqref{eq:w1}, equals zero, and a careful calculation shows 
$$
 \mathcal L(w,v) - \mathcal L(u,v) = \frac{1}{2}\|u-w\|_{L^2(\partial\Omega)}^2 \ge 0 \quad \text{for any}\quad w\in H^2(\Omega). 
$$ 
 Therefore, \eqref{eq:lsa} holds.

Next we show that the solution to \eqref{eq:1} and the solution to \eqref{eq:w1} are the same.
\begin{lemma}\label{lm2}
Let the assumptions of Theorem \ref{thm1} hold. Then $u$ is a solution to \eqref{eq:1} if and only if it is a solution to \eqref{eq:w1}.
\end{lemma}

\begin{proof}
    First, by the regularity theory of elliptic equations, there exists a unique solution $u \in H^2(\Omega)$ to \eqref{eq:1} (see for example \cite{grisvard2011elliptic}). It is simple to verify that $(u,v=0)$ is also a solution to \eqref{eq:w1}. Hence it suffices to show the uniqueness of the solution to  \eqref{eq:w1}.
   By the linearity of the problem  \eqref{eq:w1}, we need to show that for $f=0,g=0$, \eqref{eq:w1} only has a zero solution. First, from the second equation of \eqref{eq:w1} with $f=0$, we get 
    \begin{align}\label{eq:p1}
        -\Delta u = 0\quad  \text{in }\Omega,
    \end{align}
    almost everywhere. Let $h$ be an arbitrary function in $L^2(\partial \Omega)$ and $w \in H^2(\Omega)$ be the solution to equation 
    \begin{align*}
        -\Delta w = 0\quad \text{in }\Omega,\quad 
        w = h\quad \text{on } \partial\Omega,
    \end{align*}
    then with  \eqref{eq:w1}, we have
    \begin{align*}
        \int_{\partial\Omega} u h ds =0 \;  \text{for any }h \in L^2(\partial \Omega),
    \end{align*}
    which implies $u=0$ almost everywhere on $\partial\Omega$. This, together with \eqref{eq:p1}, implies $u(x)=0$ for $x\in \Omega$. 

    To show $v=0$, let $r\in L^2(\Omega)$ be an arbitrary function and we take the test function $w$ to be the solution to 
    \begin{align*}
        -\Delta w = r\quad \text{in }\Omega,\quad 
        w = 0 \quad \text{on } \partial\Omega,
    \end{align*}
    then \eqref{eq:w1} becomes 
    \begin{align}
        \int_\Omega r v dx =0\quad \text{for any }r \in L^2(\Omega),
    \end{align}
    which implies $v=0$ almost everywhere.  This finishes the proof.
\end{proof}

\subsubsection*{Proof of Theorem 1.}
If $u$ is a solution to \eqref{eq:1}, then by Lemma \ref{lm2}, $(u,v=0)$ solves \eqref{eq:w1} and thus $(u,v=0)$ is a saddle point of $\mathcal L$. Therefore, it solves \eqref{eq:dual} and thus problem \eqref{eq:is}. Conversely, if $(u,v)$ solves \eqref{eq:is}, then by the strong max-min property, the primal and dual optimal solutions $u$ and $v$ are the saddle point of $\mathcal L$. Hence $(u,v)$ solves \eqref{eq:w1}. And by Lemma \ref{lm2}, $u$ is the solution to \eqref{eq:1}.

\subsection{The existence of optimization problem \eqref{eq:co} and its equivalence to the PDE formulation}

Next, we establish the existence of the optimization problem \eqref{eq:co} and demonstrate that its solution satisfies \eqref{eq:1}. First,  we assert the Brezzi condition: 
\begin{align}\label{eq:br}
    \inf_{v\in L^2(\Omega)} \sup_{u \in H^2(\Omega)} \frac{\int_{\Omega} -\Delta u v dx}{\|u\|_{H^2(\Omega)} \|v\|_{L^2(\Omega)}} \ge \gamma, 
\end{align}
where $\gamma>0$ is a positive constant. Specifically,  for each $v$, we can set $-\Delta u = v$ in $\Omega$ and $u=0$ on $\partial\Omega$ and use the stability estimate $\|u\|_{H^2(\Omega)}\le C \|v\|_{L^2(\Omega)}$ for the Poisson problem. This leads to  inequality (\ref{eq:br}) with $\gamma =1/C$.

\begin{lemma}\label{lm4}
   Let the assumptions of Theorem \ref{thm1} hold. The optimization problem \eqref{eq:co} has a unique solution $u\in H^2(\Omega)$ that solves \eqref{eq:1}.
\end{lemma}

\begin{proof}
    Let $V_f:=\{u\in H^2(\Omega): - \Delta u =f\}$ be a  subspace of $H^2(\Omega)$. By the theory of elliptic equations, $V_f$ is not empty for $f\in L^2(\Omega)$. Taking an element $u_f \in V_f$ such that $u_f\neq 0$. Define $w:=u-u_f$ and the problem \eqref{eq:co} becomes 
    \begin{align*}
        \inf_{w\in V_0}\frac12 \int_{\partial \Omega} (u_f + w - g)^2 ds.
    \end{align*}
    Introducing the bilinear form 
    $$
    a(w_1,w_2) = \int_{\partial \Omega} w_1 \cdot w_2 ds\quad \text{and}\quad b(w,g-u_f) = \int_{\partial \Omega} w (g-u_f) ds,
    $$
    the above optimization problem can be rewritten in the form 
    \begin{align}\label{eq:opv}
        \inf_{w\in V_0} \frac12 a(w,w) - b(w,g-u_f).
    \end{align}
    Next we show $a(w_1,w_2)$ is a continuous coercive bilinear form on $V_0$. First, 
    $$a(w_1,w_2) \le \|w_1\|_{L^2(\partial\Omega)}\|w_2\|_{L^2(\partial\Omega)} \le C \|w_1\|_{H^2(\Omega)} \|w_2\|_{H^2(\Omega)}$$ 
    by the trace theorem. Moreover, by the regularity of elliptic equation 
    \begin{align*}
        &-\Delta w= 0, \text{ in }\Omega, \quad w = w|_{\partial\Omega}, \text{ on } \partial\Omega,
    \end{align*}
    we have 
 $
        \|w\|_{H^2(\Omega)} \le C \|w\|_{L^2(\partial\Omega)}. 
    $
    Hence  $a(w_1,w_2)$ is coercive on $V_0$, satisfying  
    \begin{align*}
        a(w,w) \ge \frac{1}{C} \|w\|_{H^2(\Omega)}^2.
    \end{align*}
    Due to the above property of $a(\cdot,\cdot)$ and $V_0$ being a nonempty closed and convex set, we apply the Lax-Milgram theorem to conclude that there exists a unique solution $w^* \in V_0$ to the problem \eqref{eq:opv}. Hence there exists a unique solution $u^*=w^*+u_f \in H^2(\Omega)$ to the optimization problem \eqref{eq:co}. Moreover,  the unique solution $u$ satisfies
    \begin{align*}
        a(w^*,s) = b(s,g-u_f),\quad \text{ for any } s\in V_0.
    \end{align*}
    Next we extend the space of $s$ from $V_0$ to $H^2$. Let $\ell$ be a linear functional defined on $H^2(\Omega)$ by $\ell(s):= a(w^*,s) - b(s,g-u_f)$. By the Brezzi condition \eqref{eq:br}, there exists a unique solution $v^*$ such that (see \cite[Theorem 4.2.1]{boffi2013mixed})
    \begin{align*}
        \ell(s) = -\int_\Omega (-\Delta s) v^* dx\quad \text{ for any } s\in H^2(\Omega).
    \end{align*}
    Using $u^*=w^*+u_f$, we can get 
    \begin{align*}
        a(u^*-u_f,s) - b(s,g-u_f) =- \int_\Omega (-\Delta s) v^* dx,
    \end{align*}
    which is 
    \begin{align*}
        \int_{\pa\Omega} u^* sds + \int_{\Omega}(-\Delta s) v^* dx = \int_{\pa\Omega} gs ds,
    \end{align*}
    and hence $(u^*,v^*)$ satisfies \eqref{eq:w1}. Moreover, since $\Delta w^*=0$ and $-\Delta u_f = f$, $u^*$ also satisfies the second equation of \eqref{eq:w1}. By Lemma \ref{lm2}, $u^*$ is a solution to equation \eqref{eq:1}.

\end{proof}

\section{Error analysis}
In this section, we analyze the generalization error of our method.
Let $(u_n,v_n)$ represent the solution obtained from InfSupNet after $n$ steps training and let $u$ be the precise solution to equation \eqref{eq:1}. We will derive an error estimate for the discrepancy  between $(u_n,v_n)$ and $(u,v)$, where $v=0$. 

Before conducting the error analysis, let's revisit a well-established result in the literature (refer to \cite{schechter1963lp}). Such an inequality has previously played a crucial role in the error analysis of PINNs for elliptics PDEs \cite{shin2023error,zeinhofer2023unified}).
\begin{lemma}\label{lm:l2elliptic} 
    Let $\Omega \subset\mathbb R^d$ be a bounded domain with Lipschitz boundary $\partial\Omega$. For any $w \in H^2(\Omega)$, there exists a constant $C>0$ depending only on $\Omega$ such that 
\begin{align*}
    \|w\|_{H^{1/2}(\Omega)} \le C (\|\Delta w\|_{L^2(\Omega)} + \|w\|_{L^2(\partial\Omega)}).
    \end{align*}
\end{lemma}
\begin{proof} Using the elliptic regularity inequality given in \cite[Theorem 2.1]{schechter1963lp}), we have 
\begin{align*}
    \|w\|_{H^{s}(\Omega)} \le C(\|\Delta w\|_{H^{s-2}(\Omega)} + \|w\|_{H^{s-1/2}(\partial\Omega)}+\|w\|_{H^{s-2}(\Omega)}),
\end{align*}
for any $s$ real. Taking $s=1/2$ and using the embedding $  \|v\|_{H^{-\alpha}(\Omega)} \le C\|v\|_{L^2(\Omega)}$ for $\alpha \ge 0$, the above inequality implies 
\begin{align*}
    \|w\|_{H^{1/2}(\Omega)} \le C(\|\Delta w\|_{H^{-3/2}(\Omega)} + \|w\|_{L^{2}(\partial\Omega)} + \|w\|_{H^{-3/2}(\Omega}) \le C (\|\Delta w\|_{L^2(\Omega)} + \|w\|_{L^{2}(\partial\Omega)}+ \|w\|_{L^2(\Omega)} ).
\end{align*}
Furthermore, applying Lemma 2.1  from \cite{may2013error} implies  
\begin{align*}
    \|w\|_{L^2(\Omega)} \le C (\|\Delta w\|_{H^{-2}(\Omega)} + \|w\|_{H^{-1/2}(\partial\Omega)}) \le C (\|\Delta w\|_{L^2(\Omega)} + \|w\|_{L^2(\partial\Omega)}).
\end{align*}
Combining these inequalities yield \eqref{eq:stability} as asserted. 
\end{proof}
Applying Lemma \ref{lm:l2elliptic} to $u_n-u$ gives
\begin{align}\label{eq:stability}
    \|u_n-u\|_{H^{1/2}(\Omega)} 
    &\le C ( \|\Delta (u_n-u)\|_{L^2(\Omega)} +  \|u_n-u\|_{L^2(\partial\Omega)}).
\end{align}
Hence to get the error estimate, it suffices to derive a bound on the terms on the right hand side of \eqref{eq:stability}.
\subsection{Error decomposition} 
We will break down the error into various components, encompassing the Monte Carlo sampling error, the training loss gap, and the neural network approximation error. The principal outcome is articulated in the following theorem.
\begin{theorem}\label{thm2}
    Let the assumptions of Theorem \ref{thm1} hold. Then the following estimate holds:
    \begin{align}\label{eq:est}
    \|u_n - u\|_{L^2(\partial \Omega)}^2 + \|\Delta (u_n -u) \|_{L^2(\Omega)}^2 
    + \|v_n\|_{L^2(\Omega)}^2 \le I_{NN} + I_{MC} + I_{GP}, 
    \end{align}
    where $I_{NN}$ is the neural network approximation error, given by 
    \begin{align}\label{eq:enn}
        I_{NN} = |\L(u_n,q_{\tilde v}) - \L(u_n,\tilde v)| + |\L(w_{\tilde u},v_n) - \L(\tilde u,v_n)|,
    \end{align}
    where $\tilde v = -\Delta u_n -f$, and $\tilde u$ is the solution to
    \begin{equation}\label{eq:tildeu}
    \begin{aligned} 
        -\Delta \tilde u &= f - v_n,\quad \text{in }\Omega, \\
        \tilde u &=g,\quad \text{on }\partial\Omega.
    \end{aligned}
    \end{equation}
        Here $q_{\tilde v},w_{\tilde u}$ are the best approximation of $\tilde v,\tilde u$ in $\mathcal V_\tau, \mathcal U_\theta$ (see \eqref{eq:qv} and \eqref{eq:wu} for the definitions),
    and $I_{MC}$ is the Mento-Carlo sampling error defined by 
    \begin{align}\label{eq:emc}
        I_{MC} = |\L(u_n,q_{\tilde v}) - \hat{\L}(u_n,q_{\tilde v})| + |\L(w_{\tilde u},v_n) - \hat \L (w_{\tilde u},v_n)|,
    \end{align} 
    and $I_{GP}$ is the duality gap of the training loss 
    \begin{align}\label{eq:ils}
        I_{GP} =\sup_{q\in \mathcal V_\tau} \hat{\L}(u_n,q) - \inf_{w\in \mathcal U_\theta} \hat \L(w,v_n).
    \end{align}
    Moreover, by 
     \eqref{eq:stability}, 
    inequality \eqref{eq:est} implies 
    \begin{align}\label{eq:uuvv}
        \|u_n-u\|_{H^{1/2}(\Omega)}^2 + \|v_n\|_{L^2(\Omega)}^2 \le C (I_{NN}+I_{MC}+I_{GP}).
    \end{align}
\end{theorem}

\begin{proof}
First we estimate $u_n-u$. For $u_n \in \mathcal U_\theta \subset H^2(\Omega)$, we define  $\tilde v = -\Delta u_n - f$ so that
$$ \|u_n - u\|_{L^2(\partial \Omega)}^2 + \|\Delta (u_n -u) \|_{L^2(\Omega)}^2  \le 2 \L(u_n,\tilde v).$$
Let $q_{\tilde v}$ be the best approximation of $\tilde v \in \mathcal V_\tau$ such that 
\begin{align}\label{eq:qv}
q_{\tilde v} = \arg\inf_{q\in \mathcal V_\tau} \|q-\tilde v\|_{L^2(\Omega)},
\end{align}
then 
\begin{align}\label{eq:er1}
    \L(u_n,\tilde v) &= \L(u_n,\tilde v) - \L(u_n, q_{\tilde v}) + \L(u_n, q_{\tilde v})\le |\L(u_n,\tilde v) - \L(u_n, q_{\tilde v})| + \L(u_n, q_{\tilde v}) .
\end{align}

The first term on the right side gives the first part of the approximation error \eqref{eq:enn}. Next we consider the second term and rewrite 
 $\L(u_n,q_{\tilde v})$  using $\hat \L$, which is defined in \eqref{eq:hatLL}, as
\begin{align*}
    \L(u_n,q_{\tilde v}) &= \L(u_n,q_{\tilde v}) - \hat \L(u_n,q_{\tilde v}) + \hat \L(u_n,q_{\tilde v}) \le |\L(u_n,q_{\tilde v}) - \hat \L(u_n,q_{\tilde v}) | + \sup_{q\in \mathcal V_\tau} \hat \L(u_n, q).
\end{align*}
The first term on the right of the last inequality gives the first part of the Mento-Carlo sampling error \eqref{eq:emc}, and the last term gives the first part of the duality gap \eqref{eq:ils}.
Combing the above inequality with \eqref{eq:er1}, we obtain
\begin{align}\label{eq:uest}
    \frac12 \|&u_n -u\|_{L^2(\partial\Omega)}^2 + \|\Delta (u_n -u)\|_{L^2(\Omega)}^2 \nonumber\\
    &\le |\L (u_n,\tilde v) - \L(u_n,q_{\tilde v})| +  |\L(u_n,q_{\tilde v}) - \hat \L(u_n,q_{\tilde v}) | + \sup_{q\in \mathcal V_\tau} \hat \L(u_n, q).
\end{align}

Next we estimate the difference between $v_n$ and $v=0$.
For any $v_n \in \mathcal V_\tau \subset L^2(\Omega)$, there exists a solution $\tilde u \in H^2(\Omega)$ to \eqref{eq:tildeu}.
Taking $u=\tilde u$ in $\L(u,v_n)$, we get 
\begin{align*}
    \L(\tilde u, v_n ) = - \|v_n\|_{L^2(\Omega)}^2.
\end{align*}

Let $w_{\tilde u}$ be the best approximation of $\tilde u$ in $\mathcal U_\theta$, i.e.
\begin{align}\label{eq:wu}
    w_{\tilde u} = \arg\inf_{w \in \mathcal U_\theta} \|w-\tilde u\|_{H^2(\Omega)},
\end{align}
then 
\begin{align}\label{eq:erv1}
    \L(\tilde u, v_n) &= \L(\tilde u, v_n) - \L(w_{\tilde u},v_n) + \L(w_{\tilde u},v_n) \ge -|\L(\tilde u, v_n) - \L(w_{\tilde u},v_n)| + \L(w_{\tilde u},v_n).
\end{align}

The first term on the right reflects the approximation error of the neural network. 
For the second term, we have
\begin{align*}
\L(w_{\tilde u},v_n) &= \L(w_{\tilde u},v_n) - \hat \L (w_{\tilde u},v_n) + \hat \L(w_{\tilde u},v_n)\ge - |\L(w_{\tilde u},v_n) - \hat \L (w_{\tilde u},v_n)| + \inf_{w \in \mathcal U_\theta} \L(w,v_n).
\end{align*}
Combing the above inequality with \eqref{eq:erv1},
we get the estimate 
\begin{align}\label{eq:vest}
    \|v_n\|_{L^2(\Omega)}^2 &\le |\L(\tilde u, v_n) - \L(w_{\tilde u},v_n)| +  |\L(w_{\tilde u},v_n) - \hat \L (w_{\tilde u},v_n)|
    - \inf_{w\in \mathcal U_\theta} \hat \L(w,v_n).
\end{align}
Combing \eqref{eq:uest} and \eqref{eq:vest} leads to \eqref{eq:est} and finishes the proof.
\end{proof}

\subsection{Error bounds} 
To specify, we choose the hyperbolic tangent (tanh) as the activation function for both the $\mathcal V_\tau$
 and $\mathcal U_\theta$ networks.  Initially, our focus is on the neural network approximation error $I_{NN}$ as defined in \eqref{eq:enn}. This error is connected to the approximation inaccuracies of neural networks within the function spaces $L^2(\Omega)$ and $H^2(\Omega)$.

For the first term in \eqref{eq:enn}, we have
\begin{align*}
    |\L(u_n,\tilde v) - \L(u_n,q_{\tilde v})| &= \int_\Omega (-\Delta u_n -f) \cdot (\tilde v-q_{\tilde v}) dx \nonumber\\
    &\le \int_\Omega \tilde v \cdot (\tilde v-q_{\tilde v}) dx \nonumber\\
    &\le \|\tilde v\|_{L^2(\Omega)} \|\tilde v - q_{\tilde v}\|_{L^2(\Omega)}. 
\end{align*}
From Proposition \ref{lm:mishra} (see in Appendix), there exists $\mathcal V_\tau$ (two hidden layers with width given by the proposition) such that 
\begin{align}\label{eq:tilevest}
     \|\tilde v - q_{\tilde v}\|_{L^2(\Omega)} &=\inf_{q\in \mathcal{V}_\tau} \|q - \tilde v\|_{L^2(\Omega)} \le |\Omega|^{\frac12} \inf_{q\in\mathcal V_\tau} \|q-\tilde v\|_{L^\infty(\Omega)} \le |\Omega|^\frac12 C_1 M^{-s+2} \|\tilde v \|_{W^{s-2,\infty}(\Omega)}, 
\end{align}
for any $s>2$,
where $C_1>0$ is the constant specified in Proposition \ref{lm:mishra} and depends only on $d,s,|\Omega|$, and $M$ is a constant depending on the width of the network.  Here we use the inequality $\|\tilde v\|_{L^2(\Omega)}\le |\Omega|^\frac12 \|\tilde v\|_{L^\infty(\Omega)}$. Hence 
\begin{align}\label{eq:in1}
 |\L(u_n,\tilde v) - \L(u_n,q_{\tilde v})| \le  C_1|\Omega| M^{-s+2} \|\tilde v\|_{W^{s-2,\infty}(\Omega)}^2.
\end{align}
If $\mbox{ReLU}$ were taken as the activation function for $\mathcal V_\tau$, one could use the results from \cite{guhring2020error,siegel2022optimal} to derive a similar estimate.

The second term in \eqref{eq:enn} when using equation \eqref{eq:tildeu}  is bounded by 
\begin{align}\label{eq:llw}
    |\L(w_{\tilde u}&,v_n) - \L(\tilde u,v_n)| \nonumber\\
    &=  \left|\frac12 \|w_{\tilde u}-g\|_{L^2(\partial\Omega)}^2  + \int_\Omega (-\Delta w_{\tilde u} + \Delta \tilde u) v_n dx\right| \nonumber \\
    &\le \frac12 \|w_{\tilde u}-\tilde u\|_{L^2(\partial\Omega)}^2  + \|v_n\|_{L^2(\Omega)} \|\Delta (w_{\tilde u}-\tilde u)\|_{L^2(\Omega)}\nonumber\\
    &\le C(\| w_{\tilde u}-\tilde u\|_{H^2(\Omega)}^2 + \|v_n\|_{L^2(\Omega)} \|w_{\tilde u} - \tilde u\|_{H^2(\Omega)}).
\end{align}
Here $C>0$ is the constant induced in the use of the embedding  $\|h\|_{L^2(\Omega)}, \|\Delta h\|_{L^2(\Omega)} \le C \|h\|_{H^2(\Omega)}$.
By Proposition \ref{lm:mishra}, there exists $\mathcal U_\theta$ (with two hidden layers whose size given by the proposition) such that 
\begin{align}\label{eq:tilewest}
\|w_{\tilde u} - \tilde u\|_{H^2(\Omega)} &=
    \inf_{w\in \mathcal U_\theta} \|w-\tilde u\|_{H^2(\Omega)} \le |\Omega|^\frac12 \inf_{w\in \mathcal U_\theta} \|w-\tilde u\|_{W^{2,\infty}(\Omega)} \nonumber\\
    &\le C_2 |\Omega|^\frac12  \max\{R_2^2,\log^2(\beta M^{s+d+2})\}
      M^{-s+2}\|\tilde u\|_{W^{s,\infty}(\Omega)},
\end{align}
where 
$C_2$ is a positive constant depending on $d,s,|\Omega|$, $R_2>1$ is a constant such that $ \tanh' x, \tanh ''x$ are decreasing on $[R_2,\infty)$ and $\beta$ is a constant depending on $d,k,\tilde u$ and is given by \eqref{eq:beta}. Since both functions are decreasing on $[1,\infty)$, we can take $R_2=1$. 
From \eqref{eq:tildeu}, it follows that                
\begin{align*}
    -\Delta (u-\tilde u) = v_n,\quad \text{in }\Omega, \quad 
    u-\tilde u = 0,\quad \text{on }\partial\Omega,
\end{align*}
which by the regularity theory of this elliptic problem yields the following estimate 
\begin{align*}
\|u-\tilde u\|_{W^{s,\infty}(\Omega)} \le C\|v_n\|_{W^{s-2,\infty}(\Omega)}.
\end{align*}
Hence 
\begin{align*}
    \|\tilde u\|_{W^{s,\infty}(\Omega)} \le \|\tilde u - u\|_{W^{s,\infty}(\Omega)} + \|u\|_{W^{s,\infty}(\Omega)} \le C\|v_n\|_{W^{s-2,\infty}(\Omega)} + \|u\|_{W^{s,\infty}(\Omega)} .
\end{align*}
Therefore, \eqref{eq:llw} implies 
\begin{align*}
    |\L(w_{\tilde u},&v_n) - \L(\tilde u,v_n)| \nonumber\\
    \le~& C C_2^2 |\Omega|  \max\{1,\log^4(\beta M^{s+d+2})\} M^{-2s+4}\|\tilde u\|_{W^{s,\infty}(\Omega)}^2 \nonumber\\
    &~+ CC_2 |\Omega|^\frac12  \max\{1,\log^2(\beta M^{s+d+2})\}
      M^{-s+2}\|\tilde u\|_{W^{s,\infty}} \|v_n\|_{L^2(\Omega)}\nonumber\\ 
    \le~& C C_2^2 |\Omega|  \max\{1,\log^4(\beta M^{s+d+2})\} M^{-2s+4}\|\tilde u\|_{W^{s,\infty}}^2 \nonumber\\
    &~+ CC_2 |\Omega|^\frac12  \max\{1,\log^2(\beta M^{s+d+2})\}
      M^{-s+2}(\|\tilde u\|_{W^{s,\infty}}^2 + \|v_n\|_{L^2(\Omega)}^2)\nonumber\\ 
    \le~& C(C_2 |\Omega|^\frac12 \max\{1,\log^2(\beta M^{s+d+2})\} M^{-s+2}+ 
    C_2^2 |\Omega| \max\{1, \log^4(\beta M^{s+d+2})\} M^{-2s+4}) \nonumber\\
    &\cdot (\|v_n\|_{W^{s-2,\infty}(\Omega)}^2 + \|u\|_{W^{s,\infty}(\Omega)}^2).
\end{align*}
Combing the above inequality with \eqref{eq:in1}, we get 
\begin{align}\label{eq:innp}
    I_{NN} \le C M^{-s+2} + C \max\{1,\log^4(\beta M^{s+d+2})\}M^{-2s+4},
\end{align}
where $C>0$ is a constant depending on $d,s,|\Omega|$, and $\|u\|_{W^{s,\infty}(\Omega)}, \|u_n\|_{W^{s,\infty}(\Omega)}$ and $\|v_n\|_{W^{s-2,\infty}(\Omega)}$.  Considering the fact $\lim_{M\to\infty} M^{-s+2}\log^4 (\beta M^{s+d+2}) =0$, we can set 
\begin{align}\label{eq:M0}
  M_0=\min \{M|  \log^4 (\beta M^{s+d+2}) M^{-s+2} \leq 1\},
\end{align}
so that \eqref{eq:innp} implies for $M>M_0$,
\begin{align}
    I_{NN} \le 2C M^{-s+2}.
\end{align}
Next we derive the estimate for the Monte-Carlo sampling error $I_{MC}$ defined by \eqref{eq:emc}. By \cite{caflisch1998monte},   we can get that for $N_b$ and $N$ large,
\begin{align}
    \left|\frac12 \int_{\partial\Omega} (  u_n -g)^2 ds -  \frac{|\partial\Omega|}{2N_b}\sum_{i=1}^{N_b} \left[( u_{n}(x_b^i)-g(x_b^i))^2\right]\right| 
    &\le \frac{\sigma_{1}}{2\sqrt{N_b}} |\partial\Omega| \nu_{d-1},
    \nonumber\\
    \left|\int_\Omega ((-\Delta u_n-f) q_{\tilde v}) dx - \frac{|\Omega|}{N} \sum_{i=1}^N [(-\Delta u_n(x^i) - f(x_i)) q_{\tilde v}(x^i)]\right| &\le
    \frac{\sigma_{2}}{\sqrt{N}} |\Omega|  \nu_{d},
\end{align}
 where $\nu_{d-1},\nu_d$ are the standard normal distributions in $d-1$ and $d$ dimensions.
Here $\sigma_{1}$ is the square root of the variance of $(u_n-g)^2$ given by 
\begin{align*}
    \sigma_{1}^2 = \int_{\partial\Omega} \left|(u_n-g)^2 - \frac{1}{|\partial\Omega|}\int_{\partial\Omega} (u_n-g)^2 ds' \right|^2ds,
\end{align*}
and $\sigma_2$ is the square root of the variance of $-\Delta u_n -f$ given by 
\begin{align*}
    \sigma_2^2 = \int_\Omega \left|(-\Delta u_n -f)q_{\tilde v} - \frac{1}{|\Omega|} \int_{\Omega} (-\Delta u_n -f)q_{\tilde v} dx' \right|^2 dx.
\end{align*}
By $(a-b)^2 \le 2(a^2 +b^2)$ and Jensen's inequality $\left(\frac{1}{|\Omega|}\int_\Omega w dx\right)^2 \le \frac{1}{|\Omega|}\int_\Omega w^2 dx$ for any $w\in L^2(\Omega)$, we get 
\begin{align}\label{eq:sigma1}
    \sigma_1^2 &\le  2\|u_n-g\|_{L^4(\partial\Omega)}^4 + 2 |\partial\Omega|\left(\frac{1}{|\partial\Omega|}\int_{\partial\Omega} (u_n-g)^2 ds'\right)^2 \le  4\|u_n-g\|_{L^4(\partial\Omega)}^4.
\end{align}
Similarly,
$$\sigma_2^2 \le 4\|-(\Delta u_n - f)q_{\tilde v} \|_{L^2(\Omega)}^2 \le 4\|q_{\tilde v}\|_{L^\infty(\Omega)}^2 \|-\Delta u_n -f)\|_{L^2(\Omega)}^2.$$
From \eqref{eq:tilevest}, $q_{\tilde v}$ satisfies the inequality 
\begin{align*}
    \|q_{\tilde v } - \tilde v\|_{L^\infty(\Omega)} = \inf_{v\in\mathcal V_\tau} \|v-\tilde v\|_{L^\infty(\Omega)} \le C_1M^{-s
    +2}\|\Delta u_n + f\|_{W^{s-2,\infty}(\Omega)},
\end{align*}
and hence 
\begin{align*}
    \|q_{\tilde v}\|_{L^\infty(\Omega)} \le \|q_{\tilde v } - \tilde v\|_{L^\infty(\Omega)} + \|\tilde v\|_{L^\infty(\Omega)} \le (1+C_1M^{-s+2}) \|\Delta u_n + f\|_{W^{s-2,\infty}(\Omega)},
 \end{align*}
and thus 
\begin{align}\label{eq:sigma2}
    \sigma_2^2 \le 4(1+C_1M^{-s+2})^2 \|\Delta u_n + f\|_{W^{s-2,\infty}(\Omega)}^4.
\end{align}
Therefore, combing \eqref{eq:sigma1} and \eqref{eq:sigma2} leads to
\begin{align}\label{eq:lluq}
|\L(&u_n,q_{\tilde v}) - \hat \L(u_n,q_{\tilde v})| \nonumber \le  \|u_n-g\|_{L^4(\partial\Omega)}^2 |\partial\Omega| N_b^{-\frac12} \nu_{d-1}  +  2(1+C_1M^{-s+2}) \|\Delta u_n + f\|_{W^{s-2,\infty}(\Omega)}^2|\Omega| N^{-\frac12} \nu_d.
\end{align}
Likewise, we can also bound $\L(w_{\tilde u},v_n) - \hat L(w_{\tilde u},v_n)$,
\begin{align*}
    |\L(&w_{\tilde u},v_n) - \hat \L(w_{\tilde u},v_n) | \nonumber\\
    &\le 2 \|w_{\tilde u} -g\|_{L^4(\partial\Omega)}^2 |\partial\Omega| N_b^{-\frac12} \nu_{d-1} + 4\|v_n\|_{L^2(\Omega)} \|-\Delta w_{\tilde u} -f\|_{L^\infty(\Omega)} |\Omega| N^{-\frac12} \nu_d\nonumber\\
    &\le  C \|w_{\tilde u} - \tilde u\|_{L^\infty(\Omega)}^2|\partial\Omega| N_b^{-\frac12}\nu_{d-1} +  4 \|v_n\|_{L^2(\Omega)} \|-\Delta w_{\tilde u}  + \Delta \tilde u -v_n\|_{L^\infty(\Omega)}|\Omega| N^{-\frac12}\nu_d \nonumber\\
    &\le  C \|w_{\tilde u} - \tilde u\|_{L^\infty(\Omega)}^2 |\partial\Omega| N_b^{-\frac12} \nu_{d-1} + C (\|w_{\tilde u} - \tilde u\|_{W^{2,\infty}(\Omega)}  + \|v_n\|_{L^{\infty}(\Omega)})\|v_n\|_{L^2(\Omega)} |\Omega| N^{-\frac12}\nu_d.
\end{align*}
With \eqref{eq:tilewest}, the above inequality implies 
\begin{align*}
     |&\L(w_{\tilde u},v_n) - \hat L(w_{\tilde u},v_n) | \nonumber\\
     &\le C C_2^2 
     \max\{1,\log^4(\beta M^{s+d+2})\} M^{-2s+4} (\|v_n\|_{W^{s-2,\infty}(\Omega)}^2 + \|u\|_{W^{s,\infty}}^2)  \cdot (|\partial\Omega| N_b^{-\frac12} \nu_{d-1} + 
     |\Omega|N^{-\frac12}\nu_d). 
\end{align*}
Combing this with \eqref{eq:lluq} and \eqref{eq:M0}, we obtain  for $M>M_0$,
\begin{align}
    I_{MC} \le  C(1+M^{-s+2})(N_b^{-\frac12}\nu_{d-1} + N^{-\frac12}\nu_d),
\end{align}
 where $C>0$ is a constant depending on $d,s,|\Omega|$, $\|u\|_{W^{s,\infty}(\Omega)}, \|u_n\|_{W^{s,\infty}(\Omega)}$, $\|v_n\|_{W^{s-2,\infty}(\Omega)}$.

Note that alternative numerical integration methods can be employed to attain an enhanced  error bound. A judicious selection includes utilizing  Gaussian quadrature in low dimensions and opting for the Quasi-Monte Carlo method or the Monte-Carlo method when dealing with higher dimensions \cite{robert1999monte}.

The analysis presented above enables us to arrive at a conclusion articulated in the following theorem.

\begin{theorem} 
Let $\Omega \in \mathbb{R}^d$ be a domain with smooth boundary $\partial \Omega$, and let  $u\in H^2(\Omega)$ be the solution to the Dirichlet problem (\ref{eq:1}), where $f\in W^{s,\infty}(\Omega), g\in W^{s-2,\infty}(\partial\Omega)$ with $s > 2, s\in \mathbb{N}$. For $M \geq\max \{5d^2,M_0\}$ with $M_0$ given by \eqref{eq:M0}, 
there exists two feedforward neural networks   $\mathcal{U}_\theta,\mathcal V_\tau$, each with two hidden layers,  activation function $\tanh$ and width given by $\max\{3\left[\frac{s}{2}\right]\binom{s+d-1}{s-1} + d(M-1),3\left[\frac{d+2}{2}\right]5^d M^d\}$, such that the result $(u_n,v_n) \in \mathcal U_\theta \times \mathcal V_\tau$ after $n$-th step training of InfSupNet satisfies  \begin{align}\label{eq:finalerr}
          \|&u_n-u\|_{H^{1/2}(\Omega)}^2 + \|v_n\|_{L^2(\Omega)}^2 \le O(1) (M^{-s+2}+ (N_b^{-\frac12}\nu_{d-1} + N^{-\frac12}\nu_{d})) + I_{GP},
  \end{align}
   where $\nu_{d-1},\nu_d$ are standard normal random variables in the $d-1$ and $d$ dimensions,
  $O(1)$ may depend on 
   $d,s,|\Omega|$, $\|u\|_{W^{s,\infty}(\Omega)}, \|u_n\|_{W^{s,\infty}(\Omega)}$ and $\|v_n\|_{W^{s-2,\infty}(\Omega)}$,
  but stay bounded even when $M, N$ and $N_b$ increase. 
\end{theorem}

\begin{remark} 
The error $I_{GP}$ in \eqref{eq:ils} depends on the algorithm used to train the neural networks. Suppose in each of the training step, we conduct the Gradient Descent/Ascent until reaching the local minimum/maximum, i.e. $$
u_{n+1} = \inf_{w\in \mathcal U_\theta} \hat \L(w,v_n) \quad \text{and}\quad v_{n+1} = \sup_{v\in \mathcal V_\tau}\hat \L(u_n,v),
$$
then $I_{GP}$ can be expressed as 
\begin{align*}
    I_{GP} = \hat \L(u_{n+1},v_n) - \hat \L(u_{n},v_{n+1}).
\end{align*}
Hence, if the training algorithm converges, $(u_n,v_n)\to (\hat u, \hat v)$ as $n\to\infty$, the above error will  converge to zero. Note that in practice $\hat \L(u_n,v_n)$ may not converge to zero,  still $I_{GP}$ provides a good quantification on the training error.

If $\hat L(\theta,\tau)$ is assumed to be convex in $\theta$ and concave in $\tau$, it is proved in \cite{nemirovski2004prox} that with a single step Gradient Descent-Ascend (GDA) 
the training loss gap $I_{GP}$ after $n$ training steps satisfies
\begin{align}\label{eq:lse}
    I_{GP} = \max_\tau \hat L(\theta_n,\tau) - \min_{\theta} \hat L(\theta,\tau_n) \le C \frac{|\theta^*-\theta_0|^2 + |\tau^*-\tau_0|^2}{\eta n},
\end{align}
where $C$ is an absolute constant, $\theta_0,\tau_0$ are the initial parameters, $\eta$ is the step size and $(\theta^*,\tau^*)$ is the saddle point of $\hat L(\theta,\tau)$.
Therefore, for any small $\epsilon>0$, there exists $n$ such that 
$I_{GP}<\epsilon$ after $n$ training  steps.  
For non-convex $\hat L$, it is hard to obtain a sharp convergence rate like (\ref{eq:lse}). 
Nonetheless, if $\hat L(\theta,\tau)$ is non-convex in $\theta$ but concave in $\tau$, \cite{lin2020gradient} shows the convergence to a set where $\nabla_{\theta,\tau} \hat L$ becomes small after finite training steps. 

For the the multistep GDA approach, convergence rates are  established in \cite{nouiehed2019solving}. 
For general non-convex non-concave case, convergence analysis is rather difficult, and we refer to \cite{razaviyayn2020nonconvex} for related discussions.  
\end{remark}
\begin{remark}
Based on the estimate \eqref{eq:finalerr}, it is evident that the neural network approximation error (comprising the first term on the right) 
diminishes with the increasing width of the neural network.
Simultaneously, the Monte-Carlo sampling error reduces with the sampling size $N_b$ and $N$. 
The anticipated trend is a decrease in the gap of the training loss with the progression of training iterations.
\end{remark}

\section{Experiments}

We apply our Inf-SupNet algorithm to a set of elliptic problems posed on a bounded domain. Monte Carlo methods are employed to generate training data points.
To assess the algorithm's  accuracy, we compute the relative $L^2$ error $\|u_\theta - u^*\|_{L^2}/\|u^*\|_{L^2}$, where  $u_\theta$ is the computed solution  and $u^*$ is the true solution $u^*$. The norms are computed through Monte Carlo integration on test data points, separate from the training set. 

The solution $u$ and the Lagrangian multiplier $v$ are both represented using fully connected neural networks with a $\tanh$ activation function. 
Training the Inf-SupNet involved the use of the \emph{RMSprop} gradient descent algorithm \cite{tieleman2012rmsprop}. To ensure rapid and accurate training, we set the learning rate to $0.001$ and decay it by $0.5$ every $5000$ training epochs.

\subsection{Poisson equation with Dirichlet boundary conditions}
We first show the capacity of our method in solving high dimensional Poisson equations. We first consider the domain $\Omega = [-1,1]^d$ and examine the Poisson equation with homogeneous Dirichlet boundary conditions, i.e. $g(x)=0$ for $x\in \partial\Omega$. We consider the source term 
\begin{align*}
    f(x) = \prod_{i=1}^d \cos \left(\frac{\pi}{2} x_i\right).
\end{align*}
with $d=5$ and use a neural network with 100 neurons in each hidden layer and 4 hidden layers to represent both the $\mathcal U$ and $\mathcal V$ networks. We use \emph{RMSprop} with learning rate $0.0005$ and set the learning rate to scale by 0.5 with a  step size 2000. Figure \ref{fig:dirichlet} shows that the training error reaches approximately around 0.6\% after 20,000 iterations. Moreover, according to the bottom left figure, the point-wise error is less than 0.01. This demonstrates that our algorithm provides an accurate numerical solutions even in high dimensions. Moreover, the top right figure illustrates that the Lagrangian multiplier network output $v_\tau$ is small, which agrees with the theoretical analysis. 

Examining the loss curve in Figure \ref{fig:dirichlet}, it is evident that the loss does not exhibit  monotonic decay throughout the training iterations. This behavior is inherent in the min-max optimizations, where the gap between the training loss during gradient descent and ascent that converges to zero, rather than the loss itself. However, although the loss may increase, the error consistently decreases until the loss stablizes.

\begin{figure}
	\centering
    \includegraphics[width=0.32\linewidth]{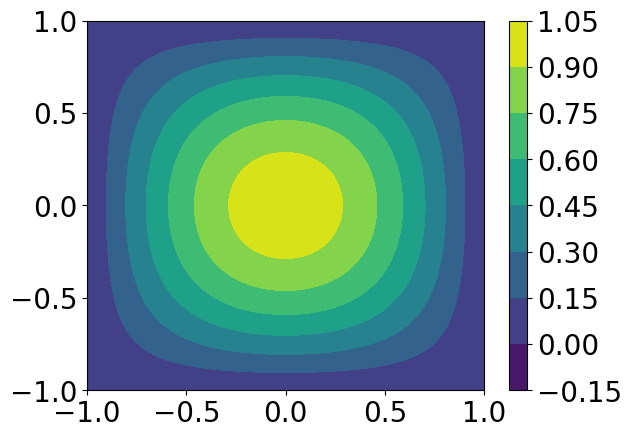} 
    \includegraphics[width=0.32\linewidth]{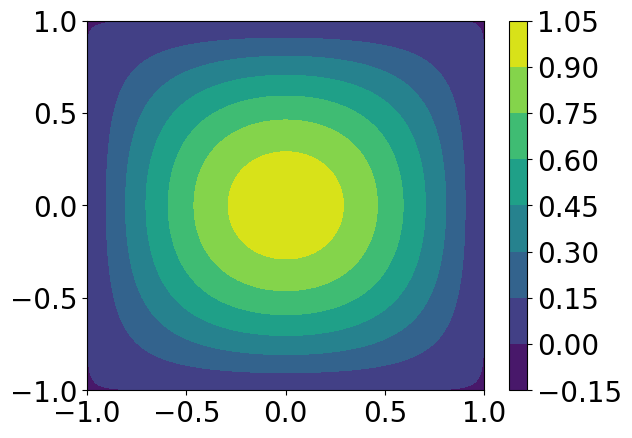} 
    \includegraphics[width=0.32\linewidth]{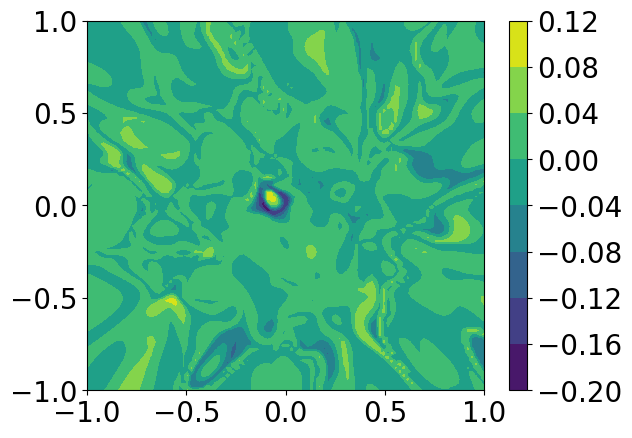}
    \includegraphics[width=0.32\linewidth]{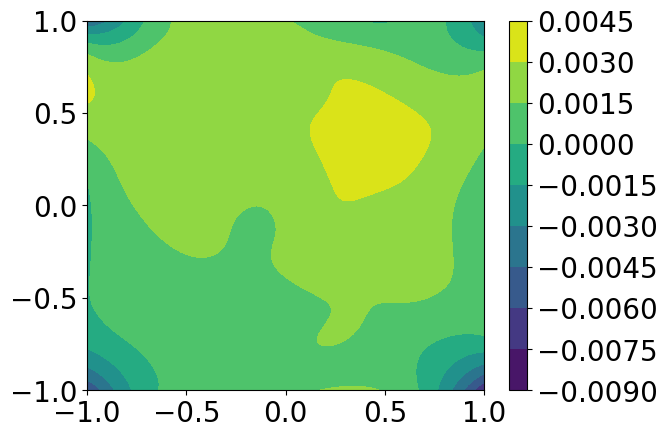}
    \includegraphics[width=0.32\linewidth]{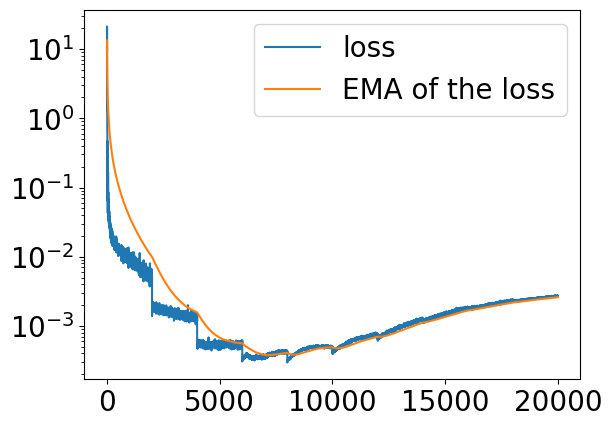}
    \includegraphics[width=0.32\linewidth]{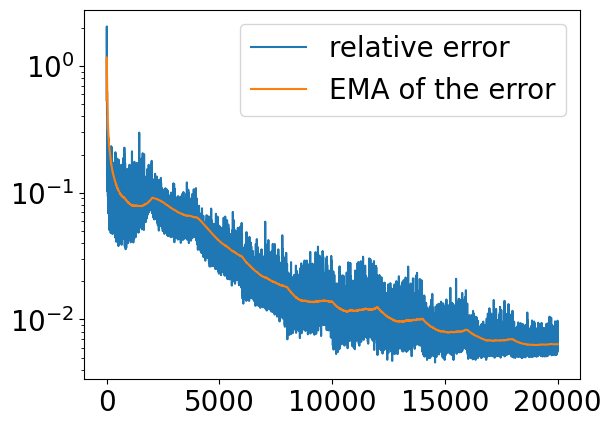}

 \caption{Numerical results for the Poisson equation with Dirichlet BC. Top left: the true solution $u^*$, top middle: the computed solution $u_\theta$; top right: the Lagrangian multiplier $v_\tau$; bottom left: the difference $u_\theta-u^*$; bottom  middle the training loss over iterations; bottom right: the test error (relative $L^2$ error with respect to the true solution) over iterations. The solutions are plotted over $x_2,x_4$ with $x_1=x_3=x_5=0$. The EMAs of the loss and error are plotted at bottom middle and right. } 
 
    \label{fig:dirichlet}
\end{figure}

To show the convergence of the algorithm, we compute the relative error with true solutions for different numbers of sampling points with $d=3$, and the relative error, presenting the results in Table \ref{tab:err1}. The table shows that the error decreases as the number of sampling points increases. The number of sampling points not only influences the Monte-Carlo integration error, but also impacts the training error. 
Hence, to get a more accurate numerical solution, we can increase the number of sampling points. However, this comes at the expense of  additional computational costs.

\begin{table}[]
    \centering
        \caption{Error table for the Poisson equation for different number of sampling points(d=3)}
    \label{tab:err1}

    \begin{tabular}{cccc}
          $N_{in}$ & $N_{bd}$ & relative $L^2$ error  \\ \hline
         2500 & 300 & $9.1638\times 10^{-3}$  \\
          5000 &  432& $4.3924\times 10^{-3}$ \\
          7500 & 516 & $3.6575\times 10^{-3}$ \\
          10000 & 600 & $2.6667\times 10^{-3}$
    \end{tabular}
\end{table}
The selection of the neural network structure is crucial for algorithm implementation. Table \ref{tab:err2} presents the error corresponding to various widths and depths (where depth represents the number of hidden layers). Elevating both the width and depth enhances the network's capability to approximate the true solution, subsequently reducing the approximation error and yielding a closer alignment with the true solution. 

\begin{table}[]
    \centering
        \caption{Error table for the Poisson equation for different neural network structures (d=3)}
    \label{tab:err2}

    \begin{tabular}{ccccc}
         width   &  error (depth=1) & error (depth=2) &error (depth=3)  \\ \hline
          10 &  $5.7629\times 10^{-2}$ & $5.7293\times 10^{-2}$ & $2.3362 \times 10^{-2}$  \\
          20 &  $1.6861\times 10^{-2}$ & $8.4221\times 10^{-3}$ & $5.3472\times 10^{-3}$ \\
          40 &  $5.9664\times 10^{-3}$ &$2.6667\times 10^{-3}$ & $2.3664\times 10^{-3}$     \end{tabular}
\end{table}

\subsection{Irregular domains}
We consider the domain $\Omega = [-1,1]^2 \setminus [0,1]^2$ and examine the Poisson equation with homogeneous Dirichlet boundary conditions, where 
\begin{align*}
    f(x) = e^{-(x+\frac12)^2 - (y+\frac12)^2}.
\end{align*}
Taking $4$ hidden layers, each with $40$ neurons in both the $\mathcal U$ and $\mathcal V$ networks, we approximate the  solution since no analytical solution is available. Also, we use the numerical solution through  finite element methods as  as a reference for the true solution $u^*$. Figure \ref{fig:Lshape} illustrates that the relative $L^2$ error reaches approximately $0.01$ after $20,000$ training iterations, with pointwise errors being  less than $0.01$. 
This observation shows that our algorithms work well in handling irregular domains.

\begin{figure}
    \centering
    \includegraphics[width=0.32\linewidth]{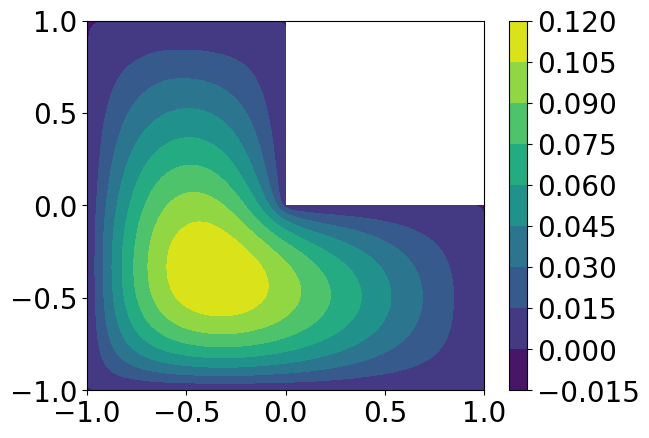}
    \includegraphics[width=0.32\linewidth]{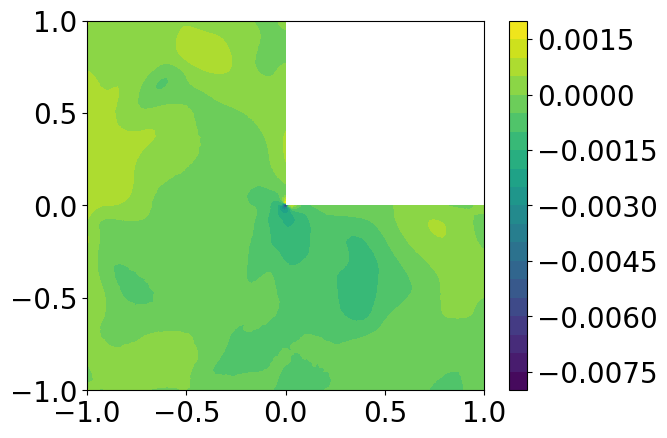}
    \includegraphics[width=0.32\linewidth]{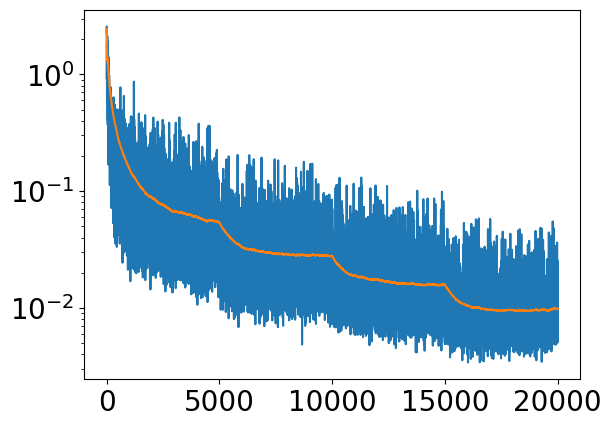}
    \caption{Numerical results on L-shaped domain. Left: solution computed by InfSupNet; middle: pointwise error with the FEM solution; Right: relative $L^2$ error during training}
    \label{fig:Lshape}
\end{figure}
\subsection{Mixed boundary conditions}
We consider the domain $\Omega=[0,1]\times [0,1]$ and the boundary set $\Gamma_D=\{(0,y)\cup (1,y) \in \partial\Omega \}$, $\Gamma_N=\{(x,0)\cup (x,1) \in \partial\Omega\}$, with boundary conditions 
\begin{align*}
    u = 0\quad \text{ on }\Gamma_D,\\
    \frac{\partial u}{\partial n} = g \quad \text{ on }\Gamma_N,
\end{align*}
where $g = \sin (5x)$. The source term in the equation $f=10 \exp(-((x-0.5)^2+(y-0.5)^2)/0.02)$. The results shown in Figure \ref{fig:mix} affirm the effectiveness of our approach in adeptly managing mixed boundary conditions.
\begin{figure}
    \centering
    \includegraphics[width=0.32\linewidth]{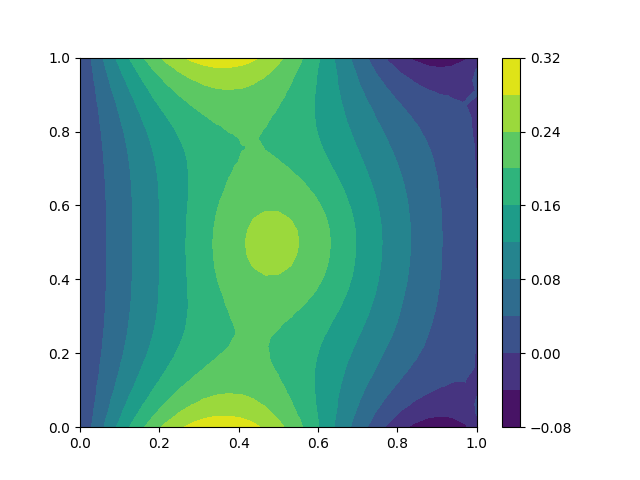}
    \includegraphics[width=0.32\linewidth]{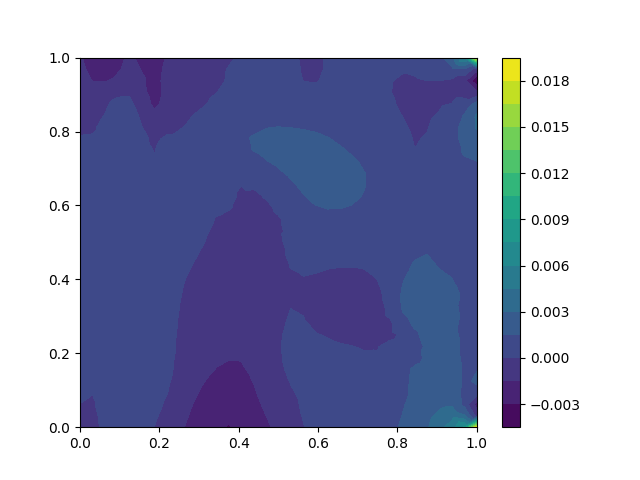}
    \includegraphics[width=0.32\linewidth]{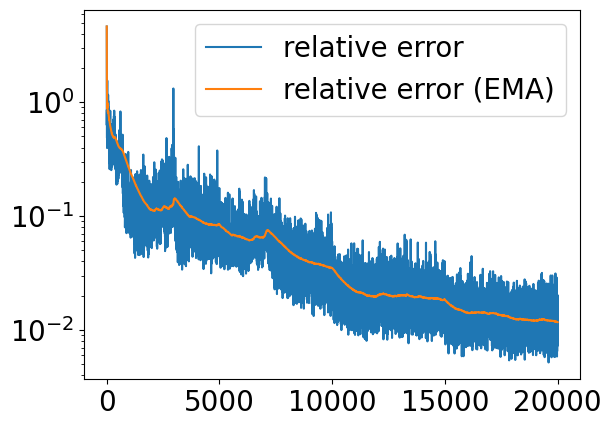}
    \caption{Numerical results for mixed boundary conditions. Left: solution computed by InfSupNet; middle: pointwise error with the FEM solution; Right: relative $L^2$ error during training}
    \label{fig:mix}
\end{figure}

\subsection{Semilinear elliptic PDE}
We address the semilinear elliptic PDE giving by 
\begin{align*}
 -\nabla \cdot (a(x)\nabla u) = f, \quad \text{ in } \Omega, \\
 u = 0,\quad \text{ on }\partial\Omega,
\end{align*}
where $a(x)=\sum_{i=1}^d x_i$, $f=f(x) =-(2d+2)\sum_{i=1}^d x_i$, and the true solution is $u^*(x)=1+|x|^2,$  with $d=5$. We randomly select 10,000 sampling points inside $\Omega$ and 400 points on the boundary $\partial\Omega$. During training, we use a fixed learning rate of 0.0002. The relative $L^2$ error between the computed solution and the true solution $u^*$ is plotted in Figure \ref{fig:senon}, revealing that the relative error reaches approximately $0.01$ after $10000$ training iterations.

\subsection{Nonlinear elliptic PDE}
We solve the nonlinear PDE 
\begin{align*}
    - \nabla \cdot (q(u) \nabla u(x)) = 0,\quad \text{in } \Omega,\\
    u = 0, \quad \text{on } \{x_1=0\}\cap \partial\Omega, \\
    u=1,\quad \text{on } \{x_1=1\} \cap \partial\Omega, \\
    \frac{\partial u}{\partial n}=0,\text{ on } \partial\Omega \backslash \{x_1=0 \text{ or } 1\},
\end{align*}
where,  $q(u)=(1+u)^m$.
The true solution $u(x) = ((2^{m+1}-1)x_1+1)^{\frac{1}{m+1}}-1$. We consider $m=2$ and $d=5$.  The relative $L^2$ error between the computed solution and the true solution $u^*$ is plotted in Figure \ref{fig:senon}, showing that the relative error reaches approximately $0.01$ after $10000$ training iterations.

\begin{figure}
    \centering
    \includegraphics[width=0.32\linewidth]{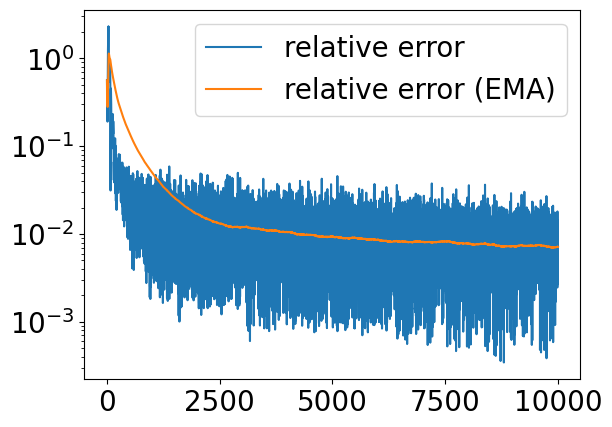}
    \includegraphics[width=0.32\linewidth]{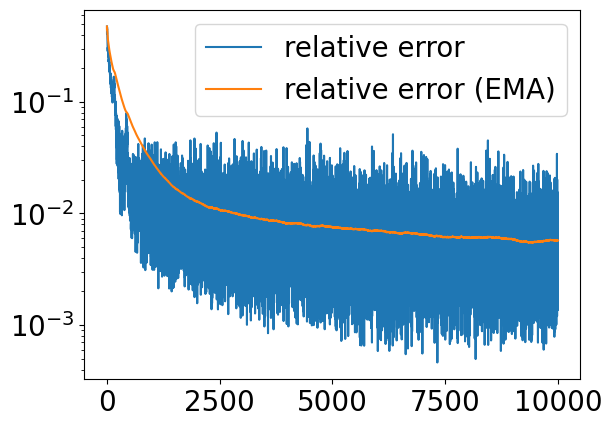}
    \caption{Numerical results for semilinear and nonlinear elliptic PDEs with $d=5$. Left: Relative error with iterations for the semilinear PDE; right: Relative error with iterations for the nonlinear PDE}
    \label{fig:senon}
\end{figure}

 \section{Concluding remarks}
We introduce Inf-SupNet, a novel method for learning solutions to elliptic PDE problems employing two deep neural networks. Our approach is grounded in two fundamental concepts. Firstly,  we establish an inf-sup reformulation of the underlying PDE problem, incorporating both the unknown variable and a Lagrangian multiplier. Secondly, we address the saddle point problem by leveraging two deep neural networks. Theoretically,  we demonstrate that the global approximation error is governed by three factors: the training error, the Monte-Carlo sampling error, and the universal approximation error inherent in neural networks. 

In our experimental evaluations, we showcase the stability and accuracy of the method when applied to high-dimensional elliptic problems on both regular and irregular domains. However, certain limitations of Inf-SupNet necessitate careful consideration. Firstly, while the method is robust with guaranteed error bounds, it may not necessarily outperform some  state-of-the-art methods. Secondly, akin to most min-max problems, the selection of the optimizer is problem-dependent.  Further investigation of training error would be highly beneficial.  Thirdly, the error bounds we obtained are based on a posterior estimate, deriving a prior estimate would enhance   the quantification of error. Future endeavours may focus on improving the efficiency of Inf-SupNet and extending its applicability to a broader class of time-dependent problems.

\section*{Acknowledgements}
This research was partially supported by the National Science Foundation under Grant DMS1812666. We would like to thank Marius Zeinhofer in pointing out a mistake in the initial draft.

\appendix
\section*{Appendix} 
We present a lemma along with its proof.
\begin{lemma} \label{lem4} 
Let $(u^*,v^*) \in H^2(\Omega)\times L^2(\Omega)$. Then $(u^*,v^*)$ is a saddle point of $\L(u,v)$ (defined in \eqref{eq:is}) if and only if the strong min-max property \eqref{eq:dual} holds.
\end{lemma}

\begin{proof}
    First,  assume $(u^*,v^*)$ is a saddle point of $\L(u,v)$, we prove \eqref{eq:dual}. For any $u,q$ we have 
    \begin{align*}
        \L(u,q) \le \sup_{v\in L^2(\Omega)} \L(u,v),\text{ for any } q\in L^2(\Omega).
    \end{align*}
    Taking the infimum gives 
    \begin{align*}
        \inf_{u\in H^2(\Omega)} \L(u,q) \le \inf_{u\in H^2(\Omega)} \sup_{v\in L^2(\Omega)} \L(u,v)
    \end{align*}
    holds for any $q \in L^2(\Omega)$. Hence
    \begin{align}\label{eq:si}
    \sup_{q \in L^2(\Omega)} \inf_{u\in H^2(\Omega)} \L(u,q) \le \inf_{u\in H^2(\Omega)} \sup_{v\in L^2(\Omega)} \L(u,v). 
    \end{align}
    From \eqref{eq:lsa}, we get 
    \begin{align*}
        \sup_{q\in L^2(\Omega)} \L(u^*,q) \le \L(u^*,v^*) \le \inf_{w\in H^2(\Omega)} \L(w,v^*).
    \end{align*}
    Therefore,
    \begin{align*}
        \inf_{u \in H^2(\Omega)} \sup_{v \in L^2(\Omega} \L(u,v) &\le \sup_{q\in L^2(\Omega)} \L(u^*,q) \le \L(u^*,v^*) \le \inf_{w\in H^2(\Omega)} \L(w,v^*) \nonumber\\
        &\le \sup_{q\in L^2(\Omega)} \inf_{w\in H^2(\Omega)} \L(w,q),
    \end{align*}
    which together with \eqref{eq:si} implies \eqref{eq:dual}. 

    Next, assume the strong min-max property \eqref{eq:dual} holds, $u^*$ is the optimal solution to the primal problem and $v^*$ is the optimal solution to the dual problem. From \eqref{eq:ll1} and \eqref{eq:ll2} we can see that $u^*$ satisfies the equation $-\Delta u^* -f = 0$, and $v^*=0$. Therefore, for any $u\in H^2(\Omega)$,
    \begin{align*}
        \L(u^*,v^*) = \frac12 \|u^*-g\|_{L^2(\partial\Omega)}^2 = \inf_{u\in H^2(\Omega)} \frac12 \|u-g\|_{L^2(\partial\Omega)}^2 \le \frac12 \|u-g\|_{L^2(\Omega)}   = L(u,v^*).
    \end{align*}
    And for any $v \in L^2(\Omega)$:
    $$\L(u^*,v^*) = \frac12 \|u^*-g\|_{L^2(\partial\Omega)}^2 + (-\Delta u^*-f,v) \ge \L(u^*,v). $$
    These two inequalities imply that $(u^*,v^*)$ is a saddle point of $\L$.
\end{proof}

Here, we present a quantitative universal approximation result; refer to \cite{de2021approximation} for a comprehensive proof. 
\begin{prop}[\cite{de2021approximation}] \label{lm:mishra}
For $h \in W^{s,\infty}(\Omega)$, and any $M\ge 5d^2$, there exists a feedforward neural network $\mathcal{U}$ with two hidden layers,  having a width given by $$\max\{3\left[\frac{s}{2}\right]\binom{s+d-1}{s-1} + d(M-1),3\left[\frac{d+2}{2}\right]5^d M^d\},$$ such that 
\begin{align}
    \inf_{\hat h \in \mathcal U}\| h-\hat h\|_{L^\infty(\Omega)} \le \frac{C_1(d,s,|\Omega|)}
    {M^s} \|h\|_{W^{s,\infty}(\Omega)}
\end{align}
for some constant $C_1=C_1(d,s,|\Omega|)>0$, and for $k=1,\ldots,s-1$, and
\begin{align}
    \inf_{\hat h\in \mathcal U} \|h - \hat h\|_{W^{k,\infty}(\Omega)} \le 
    C_2(d,k,s,|\Omega|)\frac{\max\{R_k^k,\log^k(\beta M^{s+d+2})\}}{M^{s-k}}
    \|h\|_{W^{s,\infty}(\Omega)},
\end{align}
where $C_2 = C_2(d,k,s,|\Omega|)>0$ is a positive constant only depending on $d,k,s,|\Omega|$, 
\begin{align}\label{eq:beta}
\beta =C(d,k,s)\frac{\max \{1,\|h\|_{W^{k,\infty}(\Omega)}^\frac12\}}{ \min\{1,{\|h\|_{W^{s,\infty}(\Omega)}}^\frac12\}},    
\end{align}
with $C=C(d,k,s)$ being a positive constant depending on $d,k,s$,
and $R_k$ is the lowest value such that $\tanh^{(m)}(x)$ is decreasing on $[R_k,\infty)$ for all $1\le m \le k$.
\end{prop}


\begin{thebibliography}{10}

\bibitem{barron1993universal}
Andrew~R Barron.
\newblock Universal approximation bounds for superpositions of a sigmoidal
  function.
\newblock {\em IEEE Transactions on Information theory}, 39(3):930--945, 1993.

\bibitem{beck2019machine}
Christian Beck, Weinan E, and Arnulf Jentzen.
\newblock Machine learning approximation algorithms for high-dimensional fully
  nonlinear partial differential equations and second-order backward stochastic
  differential equations.
\newblock {\em Journal of Nonlinear Science}, 29:1563--1619, 2019.

\bibitem{BN2018}
Jens Berg and Kaj Nystr{\"o}m.
\newblock A unified deep artificial neural network approach to partial
  differential equations in complex geometries.
\newblock {\em Neurocomputing}, 317:28--41, 2018.

\bibitem{bhatnagar2019prediction}
Saakaar Bhatnagar, Yaser Afshar, Shaowu Pan, Karthik Duraisamy, and Shailendra
  Kaushik.
\newblock Prediction of aerodynamic flow fields using convolutional neural
  networks.
\newblock {\em Computational Mechanics}, 64:525--545, 2019.

\bibitem{boffi2013mixed}
Daniele Boffi, Franco Brezzi, and Michel Fortin.
\newblock {\em Mixed finite element methods and applications}, volume~44.
\newblock Springer, 2013.

\bibitem{caflisch1998monte}
Russel~E Caflisch.
\newblock Monte-{C}arlo and quasi-{M}onte-{C}arlo methods.
\newblock {\em Acta Numerica}, 7:1--49, 1998.

\bibitem{cai2021physics}
Shengze Cai, Zhiping Mao, Zhicheng Wang, Minglang Yin, and George~Em
  Karniadakis.
\newblock Physics-informed neural networks ({PINN}s) for fluid mechanics: A
  review.
\newblock {\em Acta Mechanica Sinica}, 37(12):1727--1738, 2021.

\bibitem{CCLL2020}
Zhiqiang Cai, Jingshuang Chen, Min Liu, and Xinyu Liu.
\newblock Deep least-squares methods: An unsupervised learning-based numerical
  method for solving elliptic {PDE}s.
\newblock {\em Journal of Computational Physics}, 420:109707, 2020.

\bibitem{chen1993approximations}
Tianping Chen and Hong Chen.
\newblock Approximations of continuous functionals by neural networks with
  application to dynamic systems.
\newblock {\em IEEE Transactions on Neural networks}, 4(6):910--918, 1993.

\bibitem{cuomo2022scientific}
Salvatore Cuomo, Vincenzo~Schiano Di~Cola, Fabio Giampaolo, Gianluigi Rozza,
  Maziar Raissi, and Francesco Piccialli.
\newblock Scientific machine learning through physics--informed neural
  networks: Where we are and what’s next.
\newblock {\em Journal of Scientific Computing}, 92(3):88, 2022.

\bibitem{cybenko1989approximation}
George Cybenko.
\newblock Approximation by superpositions of a sigmoidal function.
\newblock {\em Mathematics of control, signals and systems}, 2(4):303--314,
  1989.

\bibitem{de2021approximation}
Tim De~Ryck, Samuel Lanthaler, and Siddhartha Mishra.
\newblock On the approximation of functions by tanh neural networks.
\newblock {\em Neural Networks}, 143:732--750, 2021.

\bibitem{de2022weak}
Tim De~Ryck, Siddhartha Mishra, and Roberto Molinaro.
\newblock Weak physics informed neural networks for approximating entropy
  solutions of hyperbolic conservation laws.
\newblock In {\em Seminar f{\"u}r Angewandte Mathematik, Eidgen{\"o}ssische
  Technische Hochschule, Z{\"u}rich, Switzerland, Rep}, volume~35, page 2022,
  2022.

\bibitem{dissanayake1994neural}
MWMG Dissanayake and Nhan Phan-Thien.
\newblock Neural-network-based approximations for solving partial differential
  equations.
\newblock {\em Communications in Numerical Methods in Engineering},
  10(3):195--201, 1994.

\bibitem{han2017deep}
Weinan E, Jiequn Han, and Arnulf Jentzen.
\newblock Deep learning-based numerical methods for high-dimensional parabolic
  partial differential equations and backward stochastic differential
  equations.
\newblock {\em Communications in Mathematics and Statistics}, 5(4):349--380,
  2017.

\bibitem{grisvard2011elliptic}
Pierre Grisvard.
\newblock {\em Elliptic problems in nonsmooth domains}.
\newblock SIAM, 2011.

\bibitem{guhring2020error}
Ingo G{\"u}hring, Gitta Kutyniok, and Philipp Petersen.
\newblock Error bounds for approximations with deep {ReLU} neural networks in
  {$W^{s, p}$} norms.
\newblock {\em Analysis and Applications}, 18(05):803--859, 2020.

\bibitem{haghighat2021physics}
Ehsan Haghighat, Maziar Raissi, Adrian Moure, Hector Gomez, and Ruben Juanes.
\newblock A physics-informed deep learning framework for inversion and
  surrogate modeling in solid mechanics.
\newblock {\em Computer Methods in Applied Mechanics and Engineering},
  379:113741, 2021.

\bibitem{jiao2021error}
Yuling Jiao, Yanming Lai, Yisu Lo, Yang Wang, and Yunfei Yang.
\newblock Error analysis of deep {R}itz methods for elliptic equations.
\newblock {\em Analysis and Applications}, 22(1):57--87, 2024.

\bibitem{khara2021neufenet}
Biswajit Khara, Aditya Balu, Ameya Joshi, Soumik Sarkar, Chinmay Hegde, Adarsh
  Krishnamurthy, and Baskar Ganapathysubramanian.
\newblock Neufenet: Neural finite element solutions with theoretical bounds for
  parametric {PDE}s.
\newblock {\em arXiv preprint arXiv:2110.01601}, 2021.

\bibitem{khara2022neural}
Biswajit Khara, Ethan Herron, Zhanhong Jiang, Aditya Balu, Chih-Hsuan Yang,
  Kumar Saurabh, Anushrut Jignasu, Soumik Sarkar, Chinmay Hegde, Adarsh
  Krishnamurthy, and Baskar Ganapathysubramanian.
\newblock Neural {PDE} solvers for irregular domains.
\newblock {\em arXiv preprint arXiv:2211.03241}, 2022.

\bibitem{kharazmi2019variational}
Ehsan Kharazmi, Zhongqiang Zhang, and George~Em Karniadakis.
\newblock Variational physics-informed neural networks for solving partial
  differential equations.
\newblock {\em arXiv preprint arXiv:1912.00873}, 2019.

\bibitem{khodayi2020varnet}
Reza Khodayi-Mehr and Michael Zavlanos.
\newblock Varnet: Variational neural networks for the solution of partial
  differential equations.
\newblock In {\em Learning for Dynamics and Control}, pages 298--307. PMLR,
  2020.

\bibitem{KLY2017}
Yuehaw Khoo, Jianfeng Lu, and Lexing Ying.
\newblock Solving parametric {PDE} problems with artificial neural networks.
\newblock {\em European Journal of Applied Mathematics}, 32(3):421--435, 2021.

\bibitem{khoo2021solving}
Yuehaw Khoo, Jianfeng Lu, and Lexing Ying.
\newblock Solving parametric {PDE} problems with artificial neural networks.
\newblock {\em European Journal of Applied Mathematics}, 32(3):421--435, 2021.

\bibitem{kutyniok2022theoretical}
Gitta Kutyniok, Philipp Petersen, Mones Raslan, and Reinhold Schneider.
\newblock A theoretical analysis of deep neural networks and parametric {PDE}s.
\newblock {\em Constructive Approximation}, 55(1):73--125, 2022.

\bibitem{lagaris1998artificial}
Isaac~E Lagaris, Aristidis Likas, and Dimitrios~I Fotiadis.
\newblock Artificial neural networks for solving ordinary and partial
  differential equations.
\newblock {\em IEEE Transactions on Neural Networks}, 9(5):987--1000, 1998.

\bibitem{li2020fourier}
Zongyi Li, Nikola Kovachki, Kamyar Azizzadenesheli, Burigede Liu, Kaushik
  Bhattacharya, Andrew Stuart, and Anima Anandkumar.
\newblock Fourier neural operator for parametric partial differential
  equations.
\newblock {\em International Conference on Learning Representations}, 2021.

\bibitem{lin2020gradient}
Tianyi Lin, Chi Jin, and Michael Jordan.
\newblock On gradient descent ascent for nonconvex-concave minimax problems.
\newblock In {\em International Conference on Machine Learning}, pages
  6083--6093. PMLR, 2020.

\bibitem{liu2023primal}
Siting Liu, Stanley Osher, Wuchen Li, and Chi-Wang Shu.
\newblock A primal-dual approach for solving conservation laws with implicit in
  time approximations.
\newblock {\em Journal of Computational Physics}, 472:111654, 2023.

\bibitem{lu2021learning}
Lu~Lu, Pengzhan Jin, Guofei Pang, Zhongqiang Zhang, and George~Em Karniadakis.
\newblock Learning nonlinear operators via deeponet based on the universal
  approximation theorem of operators.
\newblock {\em Nature Machine Intelligence}, 3(3):218--229, 2021.

\bibitem{mahmoudabadbozchelou2022nn}
Mohammadamin Mahmoudabadbozchelou, George~Em Karniadakis, and Safa Jamali.
\newblock nn-{PINN}s: Non-newtonian physics-informed neural networks for
  complex fluid modeling.
\newblock {\em Soft Matter}, 18(1):172--185, 2022.

\bibitem{mao2020physics}
Zhiping Mao, Ameya~D Jagtap, and George~Em Karniadakis.
\newblock Physics-informed neural networks for high-speed flows.
\newblock {\em Computer Methods in Applied Mechanics and Engineering},
  360:112789, 2020.

\bibitem{may2013error}
Sandra May, Rolf Rannacher, and Boris Vexler.
\newblock Error analysis for a finite element approximation of elliptic
  dirichlet boundary control problems.
\newblock {\em SIAM Journal on Control and Optimization}, 51(3):2585--2611,
  2013.

\bibitem{minakowski2023priori}
Piotr Minakowski and Thomas Richter.
\newblock A priori and a posteriori error estimates for the {D}eep {R}itz
  method applied to the {L}aplace and {S}tokes problem.
\newblock {\em Journal of Computational and Applied Mathematics}, 421:114845,
  2023.

\bibitem{mishra2022estimates}
Siddhartha Mishra and Roberto Molinaro.
\newblock Estimates on the generalization error of physics-informed neural
  networks for approximating a class of inverse problems for {PDE}s.
\newblock {\em IMA Journal of Numerical Analysis}, 42(2):981--1022, 2022.

\bibitem{mishra2023estimates}
Siddhartha Mishra and Roberto Molinaro.
\newblock Estimates on the generalization error of physics-informed neural
  networks for approximating {PDE}s.
\newblock {\em IMA Journal of Numerical Analysis}, 43(1):1--43, 2023.

\bibitem{nemirovski2004prox}
Arkadi Nemirovski.
\newblock Prox-method with rate of convergence $o(1/t)$ for variational
  inequalities with {L}ipschitz continuous monotone operators and smooth
  convex-concave saddle point problems.
\newblock {\em SIAM Journal on Optimization}, 15(1):229--251, 2004.

\bibitem{nouiehed2019solving}
Maher Nouiehed, Maziar Sanjabi, Tianjian Huang, Jason~D Lee, and Meisam
  Razaviyayn.
\newblock Solving a class of non-convex min-max games using iterative first
  order methods.
\newblock {\em Advances in Neural Information Processing Systems}, 32, 2019.

\bibitem{owhadi2015bayesian}
Houman Owhadi.
\newblock Bayesian numerical homogenization.
\newblock {\em Multiscale Modeling \& Simulation}, 13(3):812--828, 2015.

\bibitem{pinkus1999approximation}
Allan Pinkus.
\newblock Approximation theory of the {MLP} model in neural networks.
\newblock {\em Acta Numerica}, 8:143--195, 1999.

\bibitem{raissi2017physics}
Maziar Raissi, Paris Perdikaris, and George~E Karniadakis.
\newblock Physics-informed neural networks: A deep learning framework for
  solving forward and inverse problems involving nonlinear partial differential
  equations.
\newblock {\em Journal of Computational physics}, 378:686--707, 2019.

\bibitem{raissi2017inferring}
Maziar Raissi, Paris Perdikaris, and George~Em Karniadakis.
\newblock Inferring solutions of differential equations using noisy
  multi-fidelity data.
\newblock {\em Journal of Computational Physics}, 335:736--746, 2017.

\bibitem{raissi2017machine}
Maziar Raissi, Paris Perdikaris, and George~Em Karniadakis.
\newblock Machine learning of linear differential equations using {G}aussian
  processes.
\newblock {\em Journal of Computational Physics}, 348:683--693, 2017.

\bibitem{rasht2022physics}
Majid Rasht-Behesht, Christian Huber, Khemraj Shukla, and George~Em
  Karniadakis.
\newblock Physics-informed neural networks ({PINN}s) for wave propagation and
  full waveform inversions.
\newblock {\em Journal of Geophysical Research: Solid Earth},
  127(5):e2021JB023120, 2022.

\bibitem{razaviyayn2020nonconvex}
Meisam Razaviyayn, Tianjian Huang, Songtao Lu, Maher Nouiehed, Maziar Sanjabi,
  and Mingyi Hong.
\newblock Nonconvex min-max optimization: Applications, challenges, and recent
  theoretical advances.
\newblock {\em IEEE Signal Processing Magazine}, 37(5):55--66, 2020.

\bibitem{RTOP22}
Jon~A Rivera, Jamie~M Taylor, {\'A}ngel~J Omella, and David Pardo.
\newblock On quadrature rules for solving partial differential equations using
  neural networks.
\newblock {\em Computer Methods in Applied Mechanics and Engineering},
  393:114710, 2022.

\bibitem{robert1999monte}
Christian~P Robert, George Casella, and George Casella.
\newblock {\em Monte Carlo statistical methods}, volume~2.
\newblock Springer, 1999.

\bibitem{rudd2015constrained}
Keith Rudd and Silvia Ferrari.
\newblock A constrained integration ({R}itz) approach to solving partial
  differential equations using artificial neural networks.
\newblock {\em Neurocomputing}, 155:277--285, 2015.

\bibitem{schechter1963lp}
Martin Schechter. 
\newblock On {$L^p$} estimates and regularity II.
\newblock {\em Mathematica Scandinavica}, 13.1: 47--69, 1963.

\bibitem{shin2023error}
Yeonjong Shin, Zhongqiang Zhang, and George~Em Karniadakis.
\newblock Error estimates of residual minimization using neural networks for
  linear {PDE}s.
\newblock {\em Journal of Machine Learning for Modeling and Computing}, 4(4),
  2023.

\bibitem{siegel2022optimal}
Jonathan~W Siegel.
\newblock Optimal approximation rates for deep {R}e{LU} neural networks on
  {S}obolev and {B}esov spaces.
\newblock {\em Journal of Machine Learning Research}, 24(357):1--52, 2023.

\bibitem{sirignano2018dgm}
Justin Sirignano and Konstantinos Spiliopoulos.
\newblock {DGM}: A deep learning algorithm for solving partial differential
  equations.
\newblock {\em Journal of computational physics}, 375:1339--1364, 2018.

\bibitem{tieleman2012rmsprop}
Tijmen Tieleman and Geoffrey Hinton.
\newblock {RMSPROP}: divide the gradient by a running average of its recent
  magnitude. {C}oursera: Neural networks for machine learning.
\newblock {\em COURSERA Neural Networks Mach. Learn}, 17, 2012.

\bibitem{D2RM}
Carlos Uriarte, David Pardo, Ignacio Muga, and Judit Mu{\~n}oz-Matute.
\newblock A deep double deeponet method ({D2RM}) for solving partial
  differential equations using neural networks.
\newblock {\em Computer Methods in Applied Mechanics and Engineering},
  405:115892, 2023.

\bibitem{yang2021b}
Liu Yang, Xuhui Meng, and George~Em Karniadakis.
\newblock {B}-{PINN}s: Bayesian physics-informed neural networks for forward
  and inverse {PDE} problems with noisy data.
\newblock {\em Journal of Computational Physics}, 425:109913, 2021.

\bibitem{yu2018deep}
Bing Yu et~al.
\newblock The deep {R}itz method: a deep learning-based numerical algorithm for
  solving variational problems.
\newblock {\em Communications in Mathematics and Statistics}, 6(1):1--12, 2018.

\bibitem{yu2022gradient}
Jeremy Yu, Lu~Lu, Xuhui Meng, and George~Em Karniadakis.
\newblock Gradient-enhanced physics-informed neural networks for forward and
  inverse {PDE} problems.
\newblock {\em Computer Methods in Applied Mechanics and Engineering},
  393:114823, 2022.

\bibitem{zang2020weak}
Yaohua Zang, Gang Bao, Xiaojing Ye, and Haomin Zhou.
\newblock Weak adversarial networks for high-dimensional partial differential
  equations.
\newblock {\em Journal of Computational Physics}, 411:109409, 2020.

\bibitem{zeinhofer2023unified}
 Marius Zeinhofer, Rami Masri, and Kent-Andr\'e Mardal. 
 \newblock A unified framework for the error analysis of physics-informed neural networks.
 {\em arXiv preprint}, arXiv:2311.00529, 2023.
\end{thebibliography}
\end{document}